\documentclass[11pt]{amsart}

\usepackage[colorlinks=true, pdfstartview=FitV, linkcolor=blue, 
citecolor=blue]{hyperref}

\usepackage{amssymb,amsmath}
\usepackage{bbm}
\usepackage{a4wide}
\usepackage{color}

\DeclareMathAlphabet{\mymathbb}{U}{bbold}{m}{n}

\newtheorem{theorem}{Theorem}[section]
\newtheorem{prop}[theorem]{Proposition}
\newtheorem{lemma}[theorem]{Lemma}     
\newtheorem{fact}[theorem]{Fact}
\newtheorem{coro}[theorem]{Corollary}
\theoremstyle{definition}
\newtheorem{example}[theorem]{Example}
\newtheorem{remark}[theorem]{Remark}

\newcommand{\ts}{\hspace{0.5pt}}
\newcommand{\nts}{\hspace{-0.5pt}}
\newcommand{\RR}{\mathbb{R}\ts}

\newcommand{\JJ}{\mathbb{J}}
\newcommand{\ZZ}{\mathbb{Z}}
\newcommand{\NN}{\mathbb{N}}
\newcommand{\PP}{\mathbb{P}}

\newcommand{\cA}{\mathcal{A}}
\newcommand{\cC}{\mathcal{C}}
\newcommand{\cE}{\mathcal{E}}
\newcommand{\cF}{\mathcal{F}}
\newcommand{\cM}{\mathcal{M}}
\newcommand{\cO}{\mathcal{O}}

\newcommand{\mon}{\cM_{d,\preccurlyeq}}
\newcommand{\cmon}{\cC_{d,\preccurlyeq}}
\newcommand{\drei}{\cA^{_{(3)}}_{\, 0}}
\newcommand{\emb}{\cM^{\mathrm{E}}_{d}}

\newcommand{\ee}{\ts\mathrm{e}}
\newcommand{\dd}{\ts\mathrm{d}}
\newcommand{\bs}{\boldsymbol}
\newcommand{\one}{\mymathbb{1}}
\newcommand{\nix}{\mymathbb{0}}
\newcommand{\Mat}{\mathrm{Mat}}
\newcommand{\Co}{\mathrm{cent}}
\newcommand{\alg}{\mathrm{alg}}
\newcommand{\tr}{\mathrm{tr}}

\newcommand{\sgn}{\mathrm{sgn}}

\newcommand{\exend}{\hfill$\Diamond$}
\newcommand{\defeq}{\mathrel{\mathop:}=}

\newcommand{\myfrac}[2]{\frac{\raisebox{-2pt}{$#1$}}
      {\raisebox{0.5pt}{$#2$}}}

\begin{document}

\title{On equal-input and monotone Markov matrices}

\author{Michael Baake}
\address{Fakult\"at f\"ur Mathematik, Universit\"at Bielefeld, \newline
       \indent  Postfach 100131, 33501 Bielefeld, Germany}

\author{Jeremy Sumner}
\address{School of Natural Sciences, Discipline of Mathematics,
         University of Tasmania,
    \newline \indent Private Bag 37, Hobart, TAS 7001, Australia}

\begin{abstract} 
  The {practically important} classes of equal-input and of monotone
  Markov matrices are revisited, with special focus on embeddability,
  infinite divisibility, and mutual relations. Several uniqueness
  results for the {classic Markov} embedding problem are obtained in
  the process.  {To achieve our results, we need to} employ various
  algebraic and geometric tools, including commutativity, permutation
  invariance and convexity. Of particular relevance in several
  demarcation results are Markov matrices that are idempotents.
\end{abstract}


\maketitle

\section{Introduction}

Let $\cM_d$ denote the set of $d$-dimensional \emph{Markov} (or
stochastic) matrices, which are the elements of $\Mat (d,\RR)$ with
non-negative entries and all row sums equal to $1$.  Clearly, $\cM_d$
is a compact convex set, with the $d^{\ts d}$ Markov matrices with
entries in $\{0,1\}$ being its extremal points (or elements); see
\cite[Sec.~II.1]{MM} for a summary. Another classic example is the
subset of $\cM_d$ of \emph{doubly stochastic} matrices, where both the
matrix and its transpose are Markov. Here, the extremal elements are
the $d\ts !$ permutation matrices, which is known as Birkhoff's
theorem.  Concepts and methods from convexity will be {needed and}
employed throughout the manuscript; see \cite{Web} for general
background on convex structures.  \smallskip

A Markov matrix $M$ is called \emph{embeddable}
\cite{Elfving,King,Davies,CFR} if it can be written as $M=\ee^Q$ with
a \emph{rate matrix} $Q$, which is a matrix with non-negative
off-diagonal elements and vanishing row sums. A rate matrix $Q$ is
also called a \emph{Markov generator}, or simply \emph{generator},
because $\{ \ee^{t Q} : t \geqslant 0 \}$ is a semigroup of Markov
matrices with unit, and is thus a monoid; see \cite{Norris} for
general background on Markov chains in continuous time.  The set of
embeddable matrices from $\cM_d$ is denoted by $\emb$.  A Markov
matrix is called \emph{infinitely divisible} \cite{Feller,GMMS} if it
has a Markov $n$-th root for every $n\in\NN$, where $\NN$ is the set
of positive integers. It is a well-known fact \cite[Prop.~7]{King}
that a Markov matrix is embeddable if and only if it is non-singular
and infinitely divisible; see \cite[Sec.~2.3]{Higham} as well as
\cite{Guerry,ER} for related material. In the other direction of
taking powers, whenever $Q$ is a generator,
$M^{}_{\infty}\defeq \lim^{}_{t\to\infty} \ee^{t Q}$ exists, and is an
idempotent Markov matrix by \cite[Prop.~2.3{\ts}(3)]{BS}, so
$M^{2}_{\infty} = M^{}_{\infty}$.

The embedding problem goes back to Elfving \cite{Elfving} and became
prominent through the foundational paper by Kingman \cite{King}. In
the two decades following it, many abstract characterisations of
embeddable matrices were found and investigated; see \cite{Davies,BS}
and references cited there for some of the literature, as well as
\cite{HJ,Higham} for various related questions in matrix
analysis. However, beyond $d=3$, concrete criteria suitable for
real-world applications remained elusive, and the interest in the
problem diminished somewhat. New impetus then came from {theoretical
  economics and, more recently and quite vigorously, from}
bioinformatics, where precisely the embedding problem for $d=4$ is
relevant {in phylogenetics}, with significant progress {still
  being rather fresh}; see {\cite{BS,GMMS,CFR}} and references
therein.

In this context, the inference problem of discretely-observed,
continuous-time Markov chains is the starting point, compare
\cite{YR}, where natural consistency considerations led to several new
results; see \cite{Jeremy} and references therein.  Of particular
relevance, both theoretically and practically, are equal-input
matrices \cite{Steel}, which possess a powerful algebraic structure
and paved the way to some progress, also beyond this class.  In fact,
it is precisely the systematic use of some standard (and some perhaps
not quite so standard) tools from (linear) algebra that unlocks the
somewhat stuck embedding problem for progress beyond $d=2$ as needed
in the applications. It is one goal of this paper to explain some of
these techniques in action, by applying them to two particularly
important classes of Markov chains. This way, we also attempt to
convince the reader that stepping a little into algebraic methods can
be more than a little profitable. \smallskip

To avoid trivial statements, we will generally assume $d\geqslant 2$.
Let $C \in \Mat (d,\RR)$ have equal rows, each being
$(c^{}_{1}, \ldots , c^{}_{d})$, and define
$c=c^{}_{1} \nts + \ldots + c^{}_{d}$ as its \emph{parameter sum}.
Such a matrix $C$ is Markov precisely when $c=1$ together with
$c_i\geqslant 0$ for all $i$. However, it then has rank~$1$, so
$\det(C)=0$ for $d>1$. For this reason, such matrices $C$ are often of
limited interest, for instance in the context of
embeddability. Instead, consider $M^{}_{C} = (1-c) \one + C$, which is
a matrix with each row summing to $1$. Here and below, $\one$ denotes
the identity matrix. Cleary, $M^{}_{C}$ is Markov if and only if
$c_i\geqslant 0$ and $c \leqslant 1 + c_i$ for all $i$.  Following the
much-cited monograph \cite{Steel}, we call such Markov matrices
\emph{equal-input}, since they describe Markov chains where the
probability of a transition $i \to j$, for $i\ne j$, depends on $j$
only.  As detailed in \cite{Steel}, see also \cite{BS} and references
therein, they constitute an important model class in bioinformatics
and phylogenetics.  All such matrices are of the above form, and the
underlying $c$ is called its \emph{summatory parameter}.  For a given
$d$, they form another convex set, which we denote by $\cC_d$; see
Lemma~\ref{lem:convex-1} for more.

A seemingly unrelated concept, at least at first sight, is the
following.  Consider the standard simplex of $d$-dimensional
probability vectors,
\[
  \PP_d \, \defeq \, \{ (x^{}_{1}, \ldots , x^{}_{\nts d}) : \text{all
  } x_i \geqslant 0 \text{ and } x^{}_{1} \nts + \ldots + x^{}_{d} = 1
  \} \ts ,
\]
which is a compact convex set. Its extremal elements are the standard
(row) basis vectors $e_i$ with
$i \in [d\ts ] \defeq \{ 1, \ldots, d \}$. One can introduce the
partial order of \emph{stochastic monotonicity} on $\PP_d$ by saying
that $x$ is \emph{dominated} by $y$, written as $x\preccurlyeq y$,
when $\sum_{i=m}^{d} x_i \leqslant \sum_{i=m}^{d} y_i$ holds for all
$m \in [d\ts ]$; see Eq.~\eqref{eq:stoch-ord} below for an alternative
formulation. The corresponding partial order is well defined also on
the positive multiples of $\PP_d$, called \emph{level sets}, hence
extendable to convex combinations. Further, it is consistent on the
entire positive cone, where two vectors can at most be compared when
they lie in the same level set. This notion has its origin in an
important class of stochastic processes \cite{Daley} that show up in
many places in probability theory and its applications
\cite{KK,Lind,Ki}. Though practically perhaps most relevant in various
economic contexts, stochastic monotonicity induces another class of
matrices with a lot of internal structure, which is relevant in the
embedding context, as we shall explain below.
\smallskip

A Markov matrix $M = (m^{}_{ij})^{}_{1 \leqslant i,j \leqslant d}$ is
called \emph{stochastically monotone}, or \emph{monotone} for short,
when the mapping $x\mapsto x M$ preserves stochastic monotonicity;
compare \cite{Daley,KK}.  It is well-known that $M$ is monotone if and
only if its rows
$m^{}_i = (m^{}_{i1}, \ldots, m^{}_{id}) = e^{}_{i}\ts M$ are ordered
accordingly with increasing row numbers, meaning
$m^{}_i \preccurlyeq m^{}_j$ for all $i\leqslant j$. More generally,
the concept of being monotone is well defined for all non-negative
matrices with equal row sums, which simply are multiples of Markov
matrices. Then, preserving the partial order means that an inequality
in one level set is turned into the corresponding one in another.  The
monotone Markov matrices in $\cM_d$ form a closed convex set, which we
denote by $\mon$.  All elements of $\mon$ have trace $\geqslant 1$,
and the extremal points of $\mon$, as detailed in
Lemma~\ref{lem:convex} below, are the $\binom{2d-1}{d}$ monotone
Markov matrices with entries in $\{0,1\}$. Monotone Markov matrices
appear in many contexts; see \cite{KK,Lind} as well as
\cite[Ch.~3]{Ki} for examples.

A \emph{stationary vector} of $M\in\cM_d$ is any $x\in\PP_d$ with
$x M \nts = x$. Given $M\nts $, the set of all stationary vectors is
convex.  In fact, it is a subsimplex of $\PP_d$ that can be a
singleton set (as for all irreducible $M$) or larger, up to $\PP_d$
itself (for $M\nts =\one$). A Markov matrix $M$ is an
\emph{idempotent} when $M^2=M$. This means that $M$ maps any
$x\in\PP_d$ to a stationary vector in one step. Except for
$M\nts =\one$, any idempotent in $\cM_d$ has minimal polynomial
$x (x-1)$.  In particular, all idempotents in $\cM_d$ are
diagonalisable, {and all but $M=\one$ are singular.}  Markov
idempotents constitute an interesting subset of the Markov matrices in
their own right, {compare \cite[Sec.~1.6]{HM},} but also provide a
natural link between equal-input and monotone matrices. Indeed, the
above two matrix classes are intimately connected, and several of
their properties can be derived and understood interactively. In our
presentation below, idempotents will play an important role, which is
an interesting point that does not seem to have attracted {enough}
attention in the past.  \smallskip

The structure of the paper is organised as follows.  Since we build on
terminology, methods and results from \cite{BS}, we feel that a
renewed section on preliminaries is unnecessary.  While we try to make
the paper as self-contained as possible, some reference to \cite{BS}
is inevitable to avoid duplication.  Instead, we proceed by recalling
and extending some important properties of equal-input matrices in
Section~\ref{sec:ei} and then derive new results on their graded
semigroup structure (Proposition~\ref{prop:grading}), which provides
relevant insight into the properties of such Markov matrices. Then, we
determine their embeddability (Proposition~\ref{prop:ei-embed}) and
their multiplicative structure in exponential form
(Theorem~\ref{thm:BCH}).

The latter can be viewed as a non-trivial, explicit version of the
Baker--Campbell--Hausdorff (BCH) formula in this case; see the
\textsc{WikipidiA} entry on the BCH formula for a good summary of this
tool, which is particularly useful in the context of inhomogeneous
Markov chains.  Beyond the commuting case, where the BCH formula is
trivial, we are not aware of many other matrix classes with such a
favourable structure, and a better understanding of multiplicatively
closed matrix classes is needed; compare \cite{J1, J2} as well as
\cite[Ch.~7]{Steel}.

Then, we turn to the monotone matrices in Section~\ref{sec:monotone},
where we first recall some of their elementary properties and then
continue with results on embeddability (Theorem~\ref{thm:monotone-2})
and monotone generators (Proposition~\ref{prop:mon-gen}). Throughout
the discussion, idempotent matrices will naturally appear, which can
be explained by the intrinsic (pseudo{\ts}-)Poissonian structure of
infinitely divisible Markov matrices (Proposition~\ref{prop:Pois} and
Theorem~\ref{thm:divisible}).  We consider the case $d=3$ in more
detail in Section~\ref{sec:three}, where the embeddability can be
decided completely (Proposition~\ref{prop:mon-embed} and
Theorem~\ref{thm:drei}), and close with a general uniqueness result
(Theorem~\ref{thm:unique} and Corollary~\ref{coro:unique}) and some
comments on Markov roots and limiting cases in
Section~\ref{sec:unique}.

\section{Equal-input matrices and some of their
  properties}\label{sec:ei}

From now on, we always use $C$ to denote a non-negative matrix with
equal rows and parameter sum $c\geqslant 0$, where $c=0$ then implies
$C=\nix$. For any matrix $C$ of this type, one has $MC = C$ for all
$M\in\cM_d$.

\subsection{Equal-input Markov matrices}
Given such a matrix $C$, the corresponding matrix 
\begin{equation}\label{eq:equal-input}
     M^{}_{C} \, \defeq \,  (1-c) \ts \one + C 
\end{equation}
is Markov precisely when $0 \leqslant c^{}_{i}$ for all $i$ together
with $c\leqslant 1 + \min^{}_{\ts i} c^{}_i$.  It is then called an
\emph{equal-input} Markov matrix, and the set of all such matrices for
a fixed $d$ is denoted by $\cC_d$.  As $C\ne \nix$ has eigenvalues $c$
and $0$, the latter with multiplicity $d \nts - \! 1$, it is clear
that
\begin{equation}\label{eq:det-MC}
    \det (M^{}_{C}) \, = \, (1\nts -c)^{d-1} \ts .
\end{equation}
A Markov generator $Q$ is called \emph{equal-input} if it is of the
form $Q=Q^{}_{C}=C - c \ts\ts \one$, with $C$ as above and
$c\geqslant 0$ without further restrictions. Clearly, $c=0$ means
$Q=\nix$. Also here, we call $c$ the \emph{summatory parameter}, as it
will always be clear from the context whether we refer to a Markov
matrix or to a generator. Since $C^2 = c \ts\ts C$, one gets
$Q^{2}_{C} = -c \ts\ts Q^{}_{C}$, so any equal-input generator is
diagonalisable. When $c>0$, the matrix $\frac{1}{c} \ts C$ is both
Markov and an idempotent.  As we shall see, the relation
$Q^{}_{C} = M^{}_{C} - \one$ will become particularly important.

\begin{fact}\label{fact:idem-1}
  If\/ $M\in \cC_d$, with\/ $d\geqslant 2$, its summatory parameter
  satisfies\/ $c \in \bigl[ 0, \frac{d}{d-1}\bigr]$.  Moreover,
  $M\in\cC_d$ is an idempotent if and only if\/ $c=1$ or\/ $c=0$,
  where the latter case means\/ $M\nts =\one$, while all other
  idempotents are singular.
\end{fact}

\begin{proof}
  Assume that $M = M^{}_{C}$ is Markov.  Since
  $0 \leqslant c = c^{}_{1} \nts + \ldots + c^{}_{d} \leqslant 1+
  \min^{}_i c^{}_i$, the maximal value of $c$ is realised for
  $c^{}_{1} = \ldots = c^{}_{d} = \frac{c}{d}$, which gives the first
  claim. The second follows from considering the equation
\[
    (1-c) \ts \one + C \, = \, M^{}_{C} \, = \, M^{\ts 2}_{C} \, = \,
    (1-c)^2 \ts  \one + (2-c) \ts C \ts ,
\]    
which is solved by $C=\nix$ with $c=0$ or, for any $c>0$, implies
$(1-c) \ts C = \nix$ and hence $c=1$.  This leads to the two cases
stated, where $\one$ is the only non-singular idempotent by
\eqref{eq:det-MC}.
\end{proof}

Let us next look at some asymptotic properties. Here, one has
\[
  M^{\ts n}_{C} \,= \, (1-c)^{n} \ts \one +
      \frac{1 - (1-c)^{n}}{c} \, C
\]
for $n\in \NN_0$ and $c>0$ (so $C\ne \nix$). If
$\lvert 1-c \ts \rvert < 1$, we see that $M^{\ts n}_{C}$, as
$n\to\infty$, converges to the Markov matrix $\frac{1}{c} \ts
C$. Adding the case with $c=0$, but excluding $c=2$ from the
consideration (where convergence fails, occurring only for $d=2$), one
can summarise as follows.

\begin{fact}\label{fact:limit}
  Let\/ $d\geqslant 2$ and let\/ $C$ be a non-negative matrix with
  equal rows and parameter sum\/ $0 \leqslant c < 2$. Then, if the
  matrix\/ $M^{}_{C}$ from \eqref{eq:equal-input} is Markov, one has
\[
     \lim_{n\to\infty} M^{\ts n}_{C} \, = \, \begin{cases}
      \one , & \text{if\/ $c=0$}, \\  \frac{1}{c} \ts C ,
      & \text{otherwise} , \end{cases}  
\]   
where all limits are idempotents.  Here, the summatory parameter of\/
$M^{\ts n}_{C}$ is\/ $1-(1-c)^n$, which is\/ $0$ for\/ $c=0$, or
otherwise converges to\/ $1$ as\/ $n \to \infty$.  \qed
\end{fact} 

Since idempotents will show up repeatedly below, we recall the
following well-known property of Markov matrices, {which we also
prove for the reader's convenience.}

\begin{lemma}\label{lem:idem}
   For\/ $M\in\cM_d$, the following properties are equivalent.
\begin{enumerate}\itemsep=2pt
\item $M$ is a non-singular idempotent.
\item $1$ is the only eigenvalue of\/ $M$.
\item $M=\one$.
\item $M$ has minimal polynomial\/ $q(x) = x-1$.
\end{enumerate}   
\end{lemma}

\begin{proof}
  Clearly, $M$ has $1$ as an eigenvalue, because it is Markov.  When
  $M^2=M$, the only possible other eigenvalue is $0$, and (1)
  $\Rightarrow$ (2) is clear. The implications (3) $\Rightarrow$ (4)
  $\Rightarrow$ (1) are immediate, and it remains to show (2)
  $\Rightarrow$ (3).
  
  By \cite[Thm.~13.10]{Gant}, we know that the algebraic multiplicity
  of the eigenvalue $1$ agrees with the geometric one.  When no other
  eigenvalue exists, this means $M=\one$.
\end{proof}

The set $\cC_d$ of equal-input matrices is important in many
applications, see \cite{Steel,BS} and references therein, and has
interesting {and revealing} algebraic properties as follows.

\begin{fact}\label{fact:ci-rule}
  Let\/ $C$ and\/ $C{\ts}'$ be two non-negative, equal-row matrices,
  with parameter sums\/ $c$ and\/ $c{\ts}'$, such that\/ $M_C$ and\/
  $M_{C{\ts}'}$ are Markov matrices, so both lie in\/ $\cC_d$.  Then,
  one also gets\/ $M = M_C \ts M_{C{\ts}'} \in \cC_d$, where one has\/
  $M=M_{C{\ts}''}$ with\/ $C{\ts}'' = (1-c{\ts}' ) \ts C + C{\ts}'$
  and parameter sum\/ $c{\ts}'' = c + c{\ts}' - c \ts\ts c{\ts}'$.
  \qed
\end{fact}

The Markov property of $M = M_C \ts M_{C{\ts}'}$ implies
$0 \leqslant c^{\ts\prime\prime}_{i} \leqslant c^{\ts\prime\prime}
\leqslant \ts 1 + c^{\ts\prime\prime}_{i} $ for all $i$, which may not
be obvious from the formula for $C{\ts}'' \nts$. In fact, the relation
between the summatory parameters can be further analysed as follows,
where we refer to \cite{Lang} for the grading notion, {where we will
  need the group $C_2$ with two elements, here written as $\{ 1,-1 \}$
  with ordinary multiplication.}

\begin{lemma}\label{lem:ranges}
  Consider\/ $f (a,b) = a + b - a b$ for\/
  $a,b \in \nts X \! \defeq [ \ts 0,1) \cup (1,2 \ts ]$. Then, one of
  the following three cases applies.
\begin{enumerate}\itemsep=2pt
\item If\/ $\max (a,b) < 1$, one has\/
  $\ts 0 \leqslant \max (a,b) \leqslant f (a,b) < 1$.
\item If\/ $\min (a,b) > 1$, one has\/
  $\ts 0 \leqslant 2 - \min (a,b) \leqslant f (a,b) < 1$.
\item Otherwise, one has\/
  $\ts 1 < f (a,b) \leqslant \max (a,b) \leqslant 2 \ts$.
\end{enumerate}
In particular, the mapping\/ $(a,b)\mapsto f(a,b)$ turns\/ $X\nts\nts$
into a\/ $C_2$-graded, commutative monoid, with\/ $0$ as the neutral
element of $\ts X\!$ and the grading being induced by the two
connected components of $X \!$, for instance via the mapping\/
$x \mapsto \sgn (1\nts -x)$ for\/ $x\in X\! $.
\end{lemma}

\begin{proof}
  Without loss of generality, we may assume $a\leqslant b$. Also,
  observe that the function $x \mapsto x (2-x)$, on $[ \ts 0,2 \ts ]$,
  has a unique maximum at $x=1$, with value $1$, so $x (x-1) < 1$
  holds for all $x\in [\ts 0,1) \cup (1,2 \ts ]$. Now, we can look at
  the three cases as follows.
 
  When $0\leqslant a \leqslant b < 1$, where $1-b$ is positive, we
  obtain the estimate
\[
   0 \, \leqslant \, b \, \leqslant \, b + (1-b) \ts a \, = \,
   f (a,b) \, < \, b + (1-b) \, = \, 1 \ts ,
\] 
while $1 < a \leqslant b \leqslant 2$, where $1-a$ is negative, leads
to
\[
   0 \, \leqslant \ 2-a \, = \, a + 2 (1-a) \, \leqslant \,
   a + b (1-a) \, = \, f (a,b) \, \leqslant \, a + a (1-a) 
   \, = \, a (2-a) \, < \, 1 \ts .
\]  
For the remaining case, it suffices to consider
$0 \leqslant a < 1 < b \leqslant 2$, which gives
\[
   1 \, = \, b + (1-b) \, < \, b + a (1-b) \, = \, f (a,b)
   \, \leqslant \, a + b - a \, = \, b \, \leqslant \, 2 \ts ,
\]  
from which claims $(1) - (3)$ follow.

Since
$f ( f(a,b),c ) = a + b + c - ab - ac - bc + abc = f ( a, f(b,c))$,
associativity of the mapping $(a,b)\mapsto f(a,b)$ is clear, and the
$C_2$-graded monoid structure is now obvious.
\end{proof}

\begin{prop}\label{prop:grading}
  The set\/ $\cC_d$ is a monoid under matrix multiplication, with the
  subset of non-singular elements forming a submonoid. The latter is\/
  $C_2$-graded by\/ $\ts \sgn (1 \nts -c)$, where\/ $c$ is the
  summatory parameter, which matches with the grading of\/ $\nts X$
  from Lemma~\textnormal{\ref{lem:ranges}}.  When\/ $d$ is even, the
  same grading emerges from the sign of the determinant.
\end{prop}

\begin{proof}
  The semigroup property follows from Fact~\ref{fact:ci-rule}, and
  $\one \nts \in \cC_d$ shows that $\cC_d$ is a monoid. The
  non-singular matrices, which include $\one$, are closed under
  multiplication.
   
  The formula for the summatory parameter of a product from
  Fact~\ref{fact:ci-rule}, in conjunction with Lemma~\ref{lem:ranges},
  implies $c^{\ts\prime\prime}\in [\ts 0,1)$ when $c$ and
  $c^{\ts\prime}$ are either both in $[ \ts 0,1)$ or both in
  $\bigl(1, \frac{d}{d-1}\bigr]$, where the ranges follow from
  Fact~\ref{fact:idem-1}. Likewise, $c^{\ts\prime\prime} > 1$ if and
  only if $c < 1 < c^{\ts\prime}$ or $c^{\ts\prime} < 1 <
  c$. Together, this provides the claimed $C_2$-grading.
   
  For even $d$, by Eq.~\eqref{eq:det-MC}, the sign of $1 \nts -c$
  matches the sign of the determinant.
\end{proof}

\begin{remark}
Given $d\geqslant 2$, one can consider
$X_d \defeq \big\{ x\in X : x \leqslant \frac{d}{d-1} \big\}$, which
defines a submonoid of $X\nts $, which is again $C_2$-graded.  Then,
we have two successive monoid homomorphisms, namely
\[
     \cC_d  \, \xrightarrow{\quad} \, X_d
     \, \xrightarrow{\quad} \, C_2  \ts ,
\]
which summarises the grading structure. Let us mention in passing that
the grading can be extended to include the singular matrices (that is,
those with $c=1$), and thus cover all of $\cC_d$, by employing the
semigroup $\{ -1, 0 , 1\}$ instead of $C_2$.
\exend
\end{remark}

Next, we consider the set $\cC_d$ in a little more detail. We begin
with the closed subset (and semigroup) of non-trivial idempotents,
\[
  \cC_{d,1} \, \defeq \, \{ M \in \cC_d : c = 1 \}
   \, = \, \{ M\in \cC_d : M^2 = M \ne \one \}
   \, = \, \{ M \in \cC_d : \det (M) = 0 \} \ts ,
\] 
where the alternative characterisations immediately follow from
Fact~\ref{fact:idem-1} and Eq.~\eqref{eq:det-MC}. Let
$E_i \in \Mat (d,\RR)$ denote the matrix with $1$ in all positions of
column $i$ and $0$ everywhere else, which is an idempotent Markov
matrix. In fact, these matrices satisfy
\begin{equation}\label{eq:idem-ei}
  E_i \ts E_j \, = \, E_j \quad
  \text{for all } 1 \leqslant i,j \leqslant d \ts ,
\end{equation}
and it is easy to see that any convex combination
$M = \sum_{i=1}^{d} \beta_{i} E_{i}$, where all
$\beta_{i} \geqslant 0$ and
$\beta^{}_{1} \nts + \ldots + \beta^{}_{d} = 1$, is an idempotent as
well. For $d\geqslant 2$, we also define the rational matrix
\begin{equation}\label{eq:GD-def}
    G^{}_{\nts d} \, = \, \myfrac{1}{d-1} \left(
    \begin{pmatrix} 1 & \cdots & 1 \\
      \vdots & \ddots & \vdots \\ 1 & \cdots & 1 \end{pmatrix}
      - \ts \one \right) \, = \, \myfrac{1}{d-1} 
      \bigl( E^{}_{1} \nts + \ldots + E^{}_{d} - \one \bigr) ,
\end{equation}
which is the unique element of $\cC_d$ with 
maximal summatory parameter, $c=\frac{d}{d-1}$.

\begin{lemma}\label{lem:convex-1}
  The sets\/ $\cC_d$ and\/ $\cC_{d,1}$ are convex.  When\/
  $d \geqslant 2$, the extremal elements of\/ $\cC_{d,1}$ are the
  idempotent matrices\/ $E_1 , \ldots , E_d$, which remain extremal
  in\/ $\cC_d$.  There are two further extremal elements in\/ $\cC_d$,
  namely\/ $\one$ and\/ $G_d$.
\end{lemma}

\begin{proof}
  The convexity of $\cC_d$ follows from
\[
  \alpha \ts M^{}_{C} + (1-\alpha) M^{}_{C{\ts}'} \, = \, \bigl( 1 -
  (\alpha \ts c + (1-\alpha) c^{\ts\prime} \ts )\bigr) \one + \alpha
  \ts C + (1-\alpha) C{\ts}'
\]   
for $\alpha \in [ \ts 0,1]$ together with the linearity of the
summatory parameter. The convexity of the subset $\cC_{d,1}$ is then
obvious because these are the elements with $c=1$.  Clearly, the
convex combinations of the $d$ matrices $E_1, \ldots , E_d$ span
$\cC_{d,1}$. Since they are linearly independent, they must be
extremal.
   
Let $M\in \cC_d$, so $M = M^{}_{C}$, where $C$ has equal rows
$(c^{}_{1}, \ldots , c^{}_{d})$ with $c^{}_i \geqslant 0$ and
parameter sum $c \in \bigl[ 0, \frac{d}{d-1} \bigr]$, subject to the
condition $c \leqslant 1 + c^{}_{\min}$ with
$c^{}_{\min} = \min^{}_{\ts i} c^{}_i$.  
Now, we will show that $M$ is of the form
$r \ts \one + s \ts\ts G^{}_{\nts d} + \sum_i t^{}_i \ts E^{}_i$ for
some $r,s,t_i \geqslant 0$ that sum to $1$.
   
If $c\in [ \ts 0,1]$, simply choose $s=0$, $t^{}_i = c^{}_i$ and
$r=1-c$, which does the job. If $c>1$, we have $c^{}_{\min}>0$ 
{in our setting.} Choose
$s=(d-1) c^{}_{\min}$ and $t^{}_{i} = c^{}_{i} - c^{}_{\min}$, where
$s\geqslant 0$ and all $t^{}_{i} \geqslant 0$ by construction. Then,
$s+ \sum_i t^{}_{i} = c - c^{}_{\min} \leqslant 1$, so we can
complete this with $r = 1 - s - \sum_i t^{}_{i} \geqslant 0$.  It is
easy to check that this gives a convex combination with summatory
parameter $c$, and also that the sum equals $M$.
   
Consequently, the compact set $\cC_d$ is the convex hull of
$\{ \one, G^{}_{\nts d}, E^{}_{1} , \ldots , E^{}_{d} \}$, and hence
the smallest convex set that contains these $d+2$ matrices. In view of
the Krein--Milman theorem \cite[Thm.~2.6.16]{Web}, it remains to show
that they all are extremal. This is clear for $\one$, where $c=0$, and
for $G^{}_{\nts d}$, which is the only matrix in $\cC_d$ with
$c=\frac{d}{d-1}$.  Since the $E_i$ are linearly independent, but all
have $c=1$, none can be replaced by a convex combination of the other
matrices (including $\one$ and $G^{}_{\nts d}$), which completes the
argument.
\end{proof}

\begin{example}\label{ex:C2}
 For $d=2$, all Markov matrices are of equal-input type.
 The four extremal elements of $\cC_2=\cM_2$ are given by
\[
   {\renewcommand{\arraystretch}{1.2}
   \begin{array}{c|cccc} 
  c & \;\; 0 & 1 & 1 & 2 \\ \hline
  C & \;\; \nix 
    & \left(\begin{smallmatrix} 1 & 0 \\ 1 & 0 \end{smallmatrix}\right)
    & \left(\begin{smallmatrix} 0 & 1 \\ 0 & 1 \end{smallmatrix}\right)
    & \left(\begin{smallmatrix} 1 & 1 \\ 1 & 1 \end{smallmatrix}\right) \\
  M^{}_{C} & \;\; \one 
      & \left(\begin{smallmatrix} 1 & 0 \\ 1 & 0 \end{smallmatrix}\right)
     & \left(\begin{smallmatrix} 0 & 1 \\ 0 & 1 \end{smallmatrix}\right)
     & \left(\begin{smallmatrix} 0 & 1 \\ 1 & 0 \end{smallmatrix}\right)
  \end{array}}
\]
where the two Markov matrices with $c=1$ span $\cC_{2,1}$; see
\cite[Fig.~1]{BS} for an illustration.

The situation is a little more interesting for $d=3$, where
$\cC_3 \subsetneq \cM_3$, and one has
\[
   {\renewcommand{\arraystretch}{1.2}
   \begin{array}{c|ccccc} 
  c & \;\; 0 & 1 & 1 & 1 & \frac{3}{2_{\vphantom{y}}} \\ \hline
 M^{}_{C} & \;\; \one\rule[-9pt]{0pt}{26pt}
    & \left(\begin{smallmatrix} 1 & 0 & 0 \\ 1 & 0 & 0 \\
          1 & 0 & 0 \end{smallmatrix}\right)
    & \left(\begin{smallmatrix} 0 &1 & 0 \\ 0 & 1 & 0 \\
           0 & 1 & 0 \end{smallmatrix}\right)
    & \left(\begin{smallmatrix} 0 & 0 & 1 \\ 0 & 0 & 1 \\
           0 & 0 & 1 \end{smallmatrix}\right)
    & \tfrac{1}{2}\!\left(\begin{smallmatrix} 0 & 1 & 1 \\
           1 & 0 & 1 \\ 1 & 1 & 0 \end{smallmatrix}\right) 
  \end{array}}
\]
part of which will reappear in Table~\ref{tab:extremal} below.
\exend
\end{example}

\subsection{Equal-input generators and embeddability}
If $Q$ is an equal-input generator with summatory parameter $c$, so
$Q=C-c\ts\ts \one$, its exponential is
\begin{equation}\label{eq:exp-eig}
  \ee^Q \, = \, \one + \myfrac{1-\ee^{-c}}{c}\ts Q
  \, = \, \myfrac{1-\ee^{-c}}{c} \ts C + \ee^{-c}\ts \one \ts ,
\end{equation}
with $C=\nix$ for $c=0$, so the summatory parameter of $\ee^Q$ 
is always given by $\tilde{c} = 1 - \ee^{-c}$. For embeddability, one 
has the following well-known result; see \cite{King,Davies} for 
background.

\begin{lemma}\label{lem:Kendall}
  The Markov matrix\/
  $M = \left( \begin{smallmatrix} 1-a & a \\ b & 1-b \end{smallmatrix}
  \right)$ with\/ $a,b\in [ \ts 0,1]$ is embeddable if and only if\/
  $\det (M) > 0$, which is equivalent to the condition\/
  $0 \leqslant a+b <1$.  In this case, there is precisely one
  generator\/ $Q$ such that\/ $M=\ee^Q$, namely\/
  $Q = - \frac{\log(1-a-b)}{a+b} \bigl( M \nts - \one \bigr)$, which
  is an equal-input generator.
\end{lemma}

\begin{proof}
  The first statement is Kendall's theorem, see \cite[Thm.~3.1]{BS}
  for a complete formulation, while the uniqueness claim is
  established in \cite[Eq.~(5) and Cor.~3.3{\ts}]{BS}.
\end{proof}

Put differently, since $\cC_2=\cM_2$, a matrix $M\in\cM_2$ is
embeddable if and only if its summatory parameter satisfies
$0\leqslant c < 1$. The closure of the set of embeddable matrices in
$\cM_2$ consists of all infinitely divisible elements of $\cM_2$, as
we shall discuss in more detail later, in Theorem~\ref{thm:monotone-2}
and Example~\ref{ex:Pois}.

\begin{remark}\label{rem:one}
  The equation $\ee^{\ts x} = 1$ has precisely one solution $x\in\RR$,
  namely $x=0$. In contrast, $\ee^A = \one$ with $A\in\Mat (2,\RR)$
  has already infinitely many solutions, including
  $A = n \left( \begin{smallmatrix} 0 & - 2 \pi \\ 2 \pi &
      0 \end{smallmatrix} \right)$ with $n\in\ZZ$. Restricting $A$ to
  real matrices with zero row sums restores uniqueness, because one
  eigenvalue of $A$ is then $0$, hence also the second, by the
  \emph{spectral mapping theorem} (SMT), as $A$ is real. Since $A$
  must be diagonalisable by \cite[Fact~2.15]{BS}, $A=\nix$ is the only
  solution.
   
  When $A = (a_{ij})^{}_{1\leqslant i,j \leqslant d}$ is a Markov
  generator with $\ee^A = \one$, for arbitrary $d$, one has
  $1 = \det (\ee^A) = \ee^{\tr (A)}$ and thus
  $0 = \tr (A) = - \sum_{i\ne j} a_{ij}$. With $a_{ij}\geqslant 0$ for
  all $i\ne j$ by the generator property, this gives $a_{ij} = 0$ for
  all $i\ne j$, hence also $a_{ii}=0$ for all $i$.  Consequently,
  $A=\nix$ is the only generator with $\ee^A = \one$.  However,
  already for $d=3$, there are further solutions of $\ee^A = \one$
  among the real matrices with zero row sums, which is one reason why
  the embedding problem becomes {significantly} more complicated for
  $d\geqslant 3$.  \exend
\end{remark}

In general, when $Q$ is an equal-input generator, then so is
$\frac{1}{n} \ts Q$, for every $n\in\NN$. Now, we can reformulate
results from \cite{BS} and combine them with Kingman's
characterisation of embeddability via regularity in conjunction with
infinite divisibility \cite[Prop.~7]{King}. As this is compatible with
the equal-input structure, we can summarise the general situation as
follows.

\begin{prop}\label{prop:ei-embed}
  When\/ $d$ is even, $M \in \cC_d$ is embeddable if and only if\/
  $0 \leqslant c < 1$. When $d\geqslant 3$ is odd, there are further
  embeddable cases with $c>1$.
  
  For arbitrary\/ $d$ and\/ $M\in \cC_d$, the following properties are
  equivalent.
\begin{enumerate}\itemsep=2pt  
\item $M$ has positive spectrum.
\item $M$ is embeddable via an equal-input generator. 
\item $M$ is non-singular and infinitely divisible within\/ $\cC_d$. 
\item The summatory parameter of\/ $M$ satisfies\/
   $0\leqslant c < 1$. \qed
\end{enumerate}
\end{prop} 

{Note that, for $M  \in\cC_d$ with summatory parameter $c$,
  the even/odd dichotomy with the dimension 
emerges  from Eq.~\eqref{eq:det-MC}.}
A concrete example of an embeddable matrix $M\in\cC_3$ with $c>1$ is
discussed in \cite[Ex.~16]{Davies} and \cite[Ex.~4.3]{BS}.  This $M$
is infinitely divisible within $\cM_3$, but not within $\cC_3$. In
fact, since $\sqrt[n \, ]{M}$ has spectrum
$\big\{ 1, \exp \bigl( \frac{-\pi\sqrt{3}}{n} \pm \frac{\mathrm{i} \ts
  \pi}{n}\bigr) \big\}$ in this case, $M$ does not possess an $n$-th
root of equal-input type for any $n\geqslant 2$.

\begin{example}\label{ex:constant}
  Within $\cC_d$ lies the submonoid of \emph{constant-input} matrices
  \cite[Rem.~4.8]{BS}, which all are of the form
  $M_c \defeq \one + c \ts J_d$ with
  $0\leqslant c \leqslant \frac{d}{d-1}$, where
  $J_d = \frac{d-1}{d} \ts (G_d - \one)$ with $G_d$ from
  \eqref{eq:GD-def} is a constant-input generator with summatory
  parameter $1$, hence $J^{2}_{d} = - J^{}_{d}\ts $.  Clearly, $J_d$,
  as well as every constant-input matrix, is diagonalisable. If
  $c\in [\ts 0,1)$, the spectral radius of $c \ts J_d$ is $c<1$, and a
  simple calculation with $\log (\one + c \ts J_d)$ gives
\[
    M_c \, \defeq \, \one + c \ts J_d \, = \, \exp\bigl(
    -\log (1-c) \ts J_d \bigr) .
\]  
For $d$ even, by Proposition~\ref{prop:ei-embed}, no constant-input
Markov matrix with $c>1$ can be embeddable, while this changes for
$d\geqslant 3$ odd.
  
Assume that $d$ is odd and $M_c$ with $c>1$ is embeddable, so
$M=\ee^Q$, where we also have $[J_d, Q\ts ] = \nix$.  So, by
\cite[Lemma~4.10 and Fact~2.15]{BS}, $Q$ is doubly stochastic and
diagonalisable. As the eigenvalues of $M_c$ are $1$ and $1-c<0$, the
latter with multiplicity $d-1$, the spectrum of $Q$ cannot be real. In
particular, $Q$ is not symmetric.  Still, for any $a\in [\ts 0,1)$, we
get
\[
    M_a \ts M_c \, = \, \ee^{-\log (1-a)  J_d} \ee^{Q} \, = \,
    \ee^{Q - \log (1-a)  J_d} \ts ,
\]  
and $M_{f (a,c)}$ is embeddable as well, where $f$ is the function
from Lemma~\ref{lem:ranges}. 

When $c>1$ is fixed and $a$ varies in $[\ts 0,1)$, $f (a,c)$ runs
through $(1,c \ts\ts ]$. Now, the infinitely divisible elements of
$\cM_d$ form a closed subset, as follows by a standard compactness
argument via convergent subsequences.  Since $M_c$ with $c>1$ is
non-singular, we see that there is a number
$c^{}_{\max} \in (1,2 \ts ]$ such that $M_c$ is embeddable precisely
for all $c \in [\ts 0,1 ) \cup ( 1, c^{}_{\max} ]$. These
constant-input matrices form a monoid that inherits the $C_2$-grading
from Proposition~\ref{prop:grading}. For the case $d=3$, we know from
\cite[Cor.~6.6]{BS} that $c^{}_{\max} = 1+\ee^{-\pi \sqrt{3}}$. The
determination of $c^{}_{\max}>1$ for $d=2m+1$ with $m\geqslant 2$ is
an interesting open question.  \exend
\end{example}

When $d\geqslant 3$, the embedding of $M\in\cC_d$
need no longer be unique as for $d=2$, but one still has the following
property.

\begin{lemma}\label{lem:ei-unique}
  Let\/ $d\geqslant 2$ and let\/ $M\in\cC_{d}$ be embeddable. If\/ $M$
  admits a representation of the form\/ $M=\ee^Q$ with a generator\/
  $Q$ of equal-input type, the latter is unique in the sense that no
  other embedding can have an equal-input generator.
\end{lemma}

\begin{proof}
  The claim is obvious for $M=\one$, where $Q=\nix$ is the only
  {generator that solves $\one = \ee^Q$;} 
  compare Remark~\ref{rem:one}.  Next, let $Q$
  and $Q{\ts}'$ be equal-input generators, with summatory parameters
  $c$ and $c^{\ts\prime}$, where we may now assume that
  $c \ts\ts c^{\ts\prime} >0$. If $\ee^Q = \ee^{Q{\ts}'}\!$,
  Eq.~\eqref{eq:exp-eig} implies
\[
    \ee^{-c} \ts\ts \one + \myfrac{1-\ee^{-c}}{c} \, C \, = \,
     \ee^{-c^{\ts\prime}}  \one + 
       \myfrac{1-\ee^{-c^{\ts\prime}}}{c^{\ts\prime}} \, C{\ts}' .
\]
As both $C$ and $C{\ts}'$ are equal-row matrices, this can only hold
when $ \ee^{-c} = \ee^{-c^{\ts\prime}} $, hence $c=c^{\ts\prime}\nts$,
which in turn forces $C=C{\ts}'$ and thus $Q=Q{\ts}' $ as claimed.
\end{proof}

When $M\in\cC_d$ is equal-input embeddable, so $M=M^{}_{C}$ for some
{non-negative matrix} $C$ with parameter sum $0\leqslant c < 1$ by
Proposition~\ref{prop:ei-embed}, the unique generator from
Lemma~\ref{lem:ei-unique} is given by
\begin{equation}\label{eq:ei-gen}
    Q \, = \, - \frac{\log (1-c)}{c} \, (M^{}_{C}  - \one ) \ts ,
\end{equation}
meaning $Q=\nix$ for $c=0$, which is a nice extension of
Lemma~\ref{lem:Kendall}. The derivation rests on the observation that
$c<1$ is the spectral radius of $M^{}_{C} - \one$, which permits the
use of the standard branch of the matrix logarithm and its power
series.

Let us next expand on an observation made in \cite{BS}, in the context
of an effective BCH formula for embeddable equal-input matrices. Here,
one considers products of exponentials of equal-input generators,
including {the complicated case of} non-commuting ones.

\begin{theorem}\label{thm:BCH}
  Let\/ $Q $ and\/ $Q{\ts}' $ be equal-input generators, with
  summatory parameters\/ $c$ and\/ $c^{\ts\prime}$, respectively.
  Then, one has\/ $\ee^{Q} \ee^{Q{\ts}'} = \ee^{Q{\ts}''}\!$ with
\[
    Q{\ts}'' \, = \, \myfrac{c+c^{\ts\prime}}{c \ts
      \bigl(1-\ee^{-(c+c^{\ts\prime})}\bigr)}
    \bigl( \ee^{-c^{\ts\prime}} (1-\ee^{-c}) \ts Q +
    \myfrac{c}{c^{\ts\prime}}\ts (1-\ee^{-c^{\ts\prime}})
    \ts Q{\ts}' \ts \bigr) ,
\]
interpreted appropriately for\/ $c=0$ or\/ $c^{\ts\prime}=0$, where\/
$Q{\ts}''$ is again an equal-input generator.  In particular, when\/
$[Q, Q{\ts}' \ts ] =\nix$, the formula simplifies to\/
$Q{\ts}'' = Q + Q{\ts}' \nts$.
\end{theorem}

\begin{proof}
  Let $C$ and $C{\ts}'$ be the constant-row matrices underneath $Q$
  and $Q{\ts}' \nts$. Using the second identity from
  Eq.~\eqref{eq:exp-eig} in conjunction with the relation
  $C \ts C{\ts}' = c \, C{\ts}' \!$, one finds
\begin{equation}\label{eq:ei-1}
   \ee^Q \ee^{Q{\ts}'}  = \, \ee^{-(c + c^{\ts\prime}\ts )} \ts \one +
   \myfrac{\ee^{-c^{\ts\prime}} \nts (1-\ee^{-c})}{c} \, C +
   \myfrac{1-\ee^{-c^{\ts\prime}}}{c^{\ts\prime}} \, C{\ts}' .
\end{equation}
Since the summatory parameters of $\ee^Q$ and $\ee^{Q{\ts}'}$ are
$\tilde{c} = 1 - \ee^{-c}$ and
$\tilde{c}^{\ts\prime} = 1 - \ee^{-c^{\ts\prime}}$, which both lie in
$[ \ts 0,1)$, the product $\ee^Q \ee^{Q{\ts}'}$ is an equal-input
matrix that has summatory parameter
$\tilde{c}^{\ts\prime\prime} \in [ \ts 0,1)$ by
Lemma~\ref{lem:ranges}. As such, it is equal-input embeddable by
Proposition~\ref{prop:ei-embed}; see also \cite[Thm.~4.6]{BS}.
Consequently, there exists an equal-input generator
$Q{\ts}'' \! = C{\ts}'' \! - c^{\ts \prime \prime} \ts \one$ such that
\begin{equation}\label{eq:ei-2}
  \ee^Q \ee^{Q{\ts}'} \, = \, \ee^{Q{\ts}''}  = \,
  \ee^{-c^{\ts\prime\prime}} \one +
  \myfrac{1-\ee^{-c^{\ts\prime\prime}}}{c^{\ts\prime\prime}} \, C{\ts}'' ,
\end{equation}
where the last step follows once more from \eqref{eq:exp-eig}.

A comparison of \eqref{eq:ei-1} and \eqref{eq:ei-2} reveals that
equality can only hold when the summatory parameters satisfy
$c^{\ts\prime\prime} \nts = c + c^{\ts\prime}$, which then gives
\[
  C{\ts}''  = \, \myfrac{c+c^{\ts\prime}}
   {1-\ee^{-(c+c^{\ts\prime})}} \left( 
   \myfrac{\ee^{-c^{\ts\prime}} \nts (1-\ee^{-c})}{c} \, C +
   \myfrac{1-\ee^{-c^{\ts\prime}}}{c^{\ts\prime}} \, C{\ts}' \right) .
\]
Inserting $C = Q + c \ts\ts \one$ and the analogous terms for
$C{\ts}'$ and $C{\ts}''$ leads to the formula stated.

The condition $[Q, Q{\ts}' \ts ]=\nix$, which includes the case that
one generator is $\nix$, is equivalent to
$c^{\ts\prime} C = c\, C{\ts}'$, which also gives
$c^{\ts\prime} Q = c \, Q{\ts}'$. Inserting this into the formula for
$Q{\ts}''$ produces the claimed simplification after a short
calculation.
\end{proof}

In Theorem~\ref{thm:BCH}, the summatory parameters of $\ee^Q$,
$\ee^{Q{\ts}'}$ and $\ee^{Q{\ts}''}$ are always related by
\[
  \tilde{c}^{\ts\prime\prime} = \, \tilde{c} +
  \tilde{c}^{\ts\prime} \nts - \tilde{c} \ts\ts \tilde{c}^{\ts\prime}
  \, = \, f \bigl( \tilde{c} , \tilde{c}^{\ts\prime}\ts \bigr)
  \, = \, 1 - \ee^{-(c + c^{\ts\prime})} \, < \, 1 \ts ,
\]
in accordance with Fact~\ref{fact:ci-rule} and Lemma~\ref{lem:ranges}.
Various other aspects of equal- and constant-input Markov matrices
have been discussed in \cite{BS}, without, however, considering their 
connection with idempotents. We rectify this omission now by exploring
some ideas in this direction that will prove useful later.

\subsection{Markov idempotents and equal-input matrices}

When $d=2$, the only Markov idempotents are $\one$ and the equal-input
matrices $\alpha E^{}_{1} + (1-\alpha) E^{}_{2}$ with
$\alpha \in [\ts 0,1]$. This situation is deceptively simple, as one
realises already for $d=3$. Still, some general considerations are
possible, some of which can be considered as a refinement of
Fact~\ref{fact:idem-1}.

\begin{lemma}\label{lem:idem-pos-1}
  Let\/ $M = (m^{}_{ij})^{}_{1\leqslant i,j\leqslant d}$ be a Markov
  idempotent that is also a positive matrix, so\/ $m^{}_{ij} > 0$ for
  all\/ $ i,j \in [d\ts ]$.  Then, $M$ is equal-input with\/
  $M\nts =C$ for some positive equal-row matrix\/ $C$ with parameter
  sum\/ $c=1$.
\end{lemma}

\begin{proof}
  Let $m^{}_{ij}$ be a maximal element in column $j$ of $M$, so
  $m^{}_{kj} \leqslant m^{}_{ij}$ holds for all $k \in [d\ts ]$.
  Then, with $M^2 = M$, we get the inequality
\begin{equation}\label{eq:cond}
   m^{}_{ij} \, = \, \bigl( M^2 \bigr)_{ij} \, = \ts
   \sum_{k=1}^{d} m^{}_{ik} \, m^{}_{kj}
   \, \leqslant \, m^{}_{ij} \sum_{k=1}^{d} m^{}_{ik} \, = \,
   m^{}_{ij} \ts ,
\end{equation}
where we actually have equality. When all matrix elements are
positive, this is only possible if $m^{}_{kj} = m^{}_{ij}$ holds for
all $k \in [d\ts ]$, which means that column $j$ of $M$ is constant.

Since $j \in [d\ts ]$ was arbitrary, the above argument applies to 
any column of $M$, and we obtain
$M=c^{}_{1} E^{}_{1} \nts + \ldots + c^{}_{d} E^{}_{d}$, where all
$c^{}_{i} > 0$ with $c = c^{}_{1} \nts + \ldots + \ts c^{}_{d} = 1$ 
because $M$ is Markov.
\end{proof}

The statement of Lemma~\ref{lem:idem-pos-1} can also be understood via
the Perron--Frobenius theorem, as $M$ under the assumed conditions is
primitive. Then, $M^{}_{\infty} = \lim_{n\to\infty} M^n = M$ is the
projector to the unique equilibrium vector of $M$, which is
$(c^{}_{1}, \ldots , c^{}_{d})$. Note that being idempotent implies
that primitivity of $M$ becomes equivalent with positivity of $M$.

When $M$ fails to be positive, there are further cases. These are
driven by the added possibility of having equality in \eqref{eq:cond}
due to the presence of vanishing matrix elements.

\begin{example}\label{ex:idem}
  Let us look at Markov idempotents for $d=3$. When $M^2=M$ is
  positive, the complete answer is provided by
  Lemma~\ref{lem:idem-pos-1}, so we only need to analyse cases with
  zero entries. Since the set of idempotents within $\cM_d$ is closed,
  it clearly contains the simplex
\[
  \big\{ c^{}_{1} E^{}_{1} \nts + c^{}_{2} E^{}_{2} + c^{}_{3} E^{}_{3} :
  c^{}_{i} \geqslant 0 \text{ and } c^{}_{1} \nts + c^{}_{2} +
  c^{}_{3} = 1 \big\} ,
\]
in line with Fact~\ref{fact:idem-1}. This simplex includes the three
$\{ 0,1 \}$ matrices $E^{}_{1}$, $E^{}_{2}$ and $E^{}_{3}$, which are
its extremal elements. Each matrix in this simplex has a unique
equilibrium vector.

Further, one finds that the matrices
\[
  \begin{pmatrix} 1 & 0 & 0 \\ 0 & a & 1\! -\nts\nts a \\
    0 & a & 1\! -\nts\nts a  \end{pmatrix} , \quad
  \begin{pmatrix} a & 0 & 1\! -\nts\nts a \\
    0 & 1 & 0 \\ a & 0 & 1\! -\nts\nts a \end{pmatrix} ,
  \quad\text{and}\;
  \begin{pmatrix} a & 1\! -\nts\nts a & 0 \\
    a & 1\! -\nts\nts a & 0 \\ 0 & 0 & 1 \end{pmatrix}
\]
are idempotents for all $a\in [\ts 0,1]$. For each matrix, its
equilibrium vectors form a $1$-simplex, hence with two extremal
vectors. By choosing $a=0$ or $a=1$, we obtain six further $\{ 0,1 \}$
matrices. So far, all nine of them are equal-input or blockwise
equal-input, possibly after a state permutation, and the only missing
$\{ 0,1 \}$ Markov idempotent is $M=\one$.

Next, again for any $a\in [\ts 0,1]$, the matrices
\[
  \begin{pmatrix} 1 & 0 & 0 \\ 0 & 1 & 0 \\
    a & 1\! -\nts\nts a & 0  \end{pmatrix} , \quad
  \begin{pmatrix} 1 & 0 & 0 \\
    a & 0 & 1\! -\nts\nts a \\ 0 & 0 & 1 \end{pmatrix} ,
  \quad\text{and}\;
  \begin{pmatrix} 0 & a & 1\! -\nts\nts a \\
    0 & 1 & 0 \\ 0 & 0 & 1 \end{pmatrix}
\]
are Markov idempotents. They do not produce new $\{ 0,1 \}$
matrices. Also, for $0<a<1$, they are not of equal-input form, not
even blockwise, which means that more complicated cases do occur. We
leave it as an exercise to the interested reader to verify that we
have covered all cases for $d=3$, and to analyse them further.  \exend
\end{example}

{At this point, it seems worthwhile to recall the structure of general
  Markov idempotents, where we follow \cite[Sec.~1.6]{HM}.  Given any
  idempotent $M\in\cM_d$ of rank $r$, where we must have
  $1\leqslant r \leqslant d$, there is a partition of
  $[d] = \{ 1, 2, \ldots, d \}$ of the form
\begin{equation}\label{eq:part-d}
     [d] \, = \, Z \,{\dot\cup}\, K^{}_{1} 
     \,{\dot\cup}\, \ldots \, {\dot\cup} \, K^{}_{r}
\end{equation}
with $Z = \{ \ell : \text{the $\ell\ts$th column of $M$ is } 0
\}$. Without any further specification of the $K_i$, we know from the
definition of $Z$ that the matrix elements of $M$ satisfy
$m^{}_{ij} = 0$ for all $i\in [d]$ and every $j\in Z$.}  {Further,
given some subset $K \subseteq [d]$, we follow standard notation and
use $M\big|_{K\times K}$ for the restriction of $M$ to indices from
$K$.}  {Now, we can reformulate \cite[Thm.~1.16]{HM} as follows.

\begin{theorem}\label{thm:idem}
   Let\/ $M = (m^{}_{ij})^{}_{1\leqslant i,j \leqslant d}\in \cM_d$ be an 
   idempotent of rank\/ $r$. Then, there is a partition of\/ 
   $[d]$ as in \eqref{eq:part-d} with the following properties.
 \begin{enumerate}\itemsep=2pt
 \item For all\/ $s\in [r]$, one has\/ $m^{}_{ij} = 0$ for all\/
     $i\in K_s$ and\/ $j \in [d]\setminus K_s$.
 \item For all\/ $s\in[r]$, the restriction\/ $M{\big |}_{K_s \times K_s}$
     is a Markov matrix with equal, positive rows.
 \item For all\/ $i \in Z$ and\/ $s\in[r]$, and then
     every\/ $k,\ell \in K_s$, one has\/
     $ m^{}_{ik} m^{}_{k \ell} = m^{}_{i \ell} m^{}_{kk}$, where\/
     $m^{}_{kk} m^{}_{k\ell} \ne 0$ due to $(2)$.
 \end{enumerate}  
 Conversely, every Markov matrix\/ $M\in\cM_d$ with a partition 
 of\/ $[d]$ as in \eqref{eq:part-d} with
 these three properties is an idempotent.  \qed
\end{theorem}
}

%
%
%


{Specialising this classification result to $\{ 0, 1 \}$ Markov
matrices gives the following consequence, the explicit derivation of
which we leave to the interested reader.} {It can also be derived
directly by a careful analysis of $\{ 0, 1\}$ Markov matrices.}

\begin{coro}
  After possibly performing an appropriate state permutation, any
  idempotent\/ $\{ 0,1 \}$ Markov matrix appears in a block form where
  each block is an equal-input matrix with a single column of\/ $1$s.
  \qed
\end{coro}

Simple examples, {with an indication of the blocks
according to Theorem~\ref{thm:idem}{\ts}(2),}
 include
\[
  \begin{pmatrix} 
  \boxed{1} & 0 & 0 \\ 1 & 0 & 0 \\
    0 & 0 & \boxed{1} \end{pmatrix} , \quad
   \begin{pmatrix} 0 & 1 & 0 \\ 0 & \boxed{1} & 0 \\
    0 & 0 & \boxed{1} \end{pmatrix} 
    \quad\text{or}\quad
  \begin{pmatrix} \boxed{1} & 0 & 0 & 0 \\ 0 & 0 & 1 & 0 \\
    0 & 0 & \boxed{1} & 0 \\ 1 & 0 & 0 & 0 \end{pmatrix}
  \sim \begin{pmatrix} 1 & 0 & 0 & 0 \\ 1 & 0 & 0 & 0 \\
    0 & 0 & 1 & 0 \\ 0 & 0 & 1 & 0 \end{pmatrix},
\]
where the similarity in the third case is under the obvious
state permuation. On the other hand, the matrices
\[
  \begin{pmatrix} 1 & 0 & 0 \\ 1 & 0 & 0 \\
    0 & 1 & 0 \end{pmatrix} \quad\text{and}\quad
  \begin{pmatrix}1 & 0 & 0 & 0 \\ 1 & 0 & 0 & 0 \\
    0 & 1 & 0 & 0 \\ 0 & 0 & 1 & 0 \end{pmatrix}
\]
fail to be idempotent, for instance.  Let us now analyse the second
class of matrices mentioned in the Introduction before we return to
this type of structure.

\section{Monotone Markov matrices and
  embeddability}\label{sec:monotone}

Let us begin this section with a formalisation of some of our previous
recollections. To this end, we follow \cite{KK} and employ the
lower-triangular matrix $T\in \Mat (d,\RR)$ given by
\begin{equation}\label{eq:T-def}
    T \, = \, \begin{pmatrix} 
    1 &  & \nts\bs{0}\ts \\
    \vdots & \ddots \\
    1 & \cdots & 1 \end{pmatrix} ,
\end{equation}
together with its inverse, $T^{-1} \nts$, which has entries $1$ on the
diagonal, $-1$ on the first subdiagonal, and $0$ everywhere else. With
$T$ and vectors $x,y \in \PP_d$, one has the equivalence
\begin{equation}\label{eq:stoch-ord}
     x \, \preccurlyeq  \, y \quad \Longleftrightarrow \quad
     x \ts T \, \leqslant \, y \ts T ,
\end{equation}
where the inequality on the right means that it is satisfied
element-wise. 

\subsection{Monotone Markov matrices}
If $E_{(i, j)} \in \Mat (d,\RR)$ denotes the elementary matrix with a
single $1$ in position $(i,j)$ and $0$ everywhere else, which results
in
\begin{equation}\label{eq:elem-matrix}
  E^{}_{(k, \ell)} \ts E^{}_{(m, n)} \, = \,
  \delta^{}_{\ell, m} \ts E^{}_{(k, n)} \ts ,
\end{equation}
one obtains the relations
\[
    E_{(i, j)} T \, = \, E_{(i, 1)} \nts + \ldots + E_{(i, j)}
    \quad \text{and} \quad T^{-1} E_{(i, j)} \, = \,
    E_{(i, j)} - E_{(i+1,j)} \ts ,
\]
where $E_{(d+1,j)} \defeq \nix$. Further, we call a column vector
$v = (v^{}_{1}, \ldots , v^{}_{d})^T$ \emph{non-decreasing} if
$v_i \leqslant v_{i+1}$ holds for all
$i \in [d \nts - \! 1]$. Now, we can characterise
monotone matrices as follows.

\begin{fact}\label{fact:monotone}
  For a Markov matrix\/ $M\in \cM_d$, the following statements are
  equivalent.
\begin{enumerate}\itemsep=2pt
\item The matrix\/ $M$ is monotone.
\item The mapping\/ $x\mapsto xM$ preserves the partial order\/
  $\preccurlyeq$ on the positive cone.
\item One has\/ $T^{-1} \nts M \ts T \geqslant \nix$, understood
  element-wise, with\/ $T$ as in Eq.~\eqref{eq:T-def}.
\item Whenever\/ $v$ is a non-decreasing vector, $Mv$ is also
  non-decreasing.
\end{enumerate}   
Further, the same equivalences hold for any non-negative\/
$B \in \Mat (d,\RR)$ with equal row sums.
\end{fact}

\begin{proof}
  By our definition, (1) is equivalent to preserving $\preccurlyeq$ on
  $\PP_d$, which clearly extends to all level sets $\alpha \ts\PP_d$
  with $\alpha>0$, so (1) $\Longleftrightarrow$ (2) is clear. The
  equivalences (2) $\Longleftrightarrow$ (3) $\Longleftrightarrow$ (4)
  are now an immediate consequence of \cite[Thm.~1.1]{KK}.
  
  The case $B=\nix$ is trivial. When $B\ne \nix$, with elements
  $b_{ij}$, is a non-negative matrix with equal row sums, meaning that
  $\sum_{j=1}^{d} b_{ij} = b > 0$ for all $i \in [d\ts ]$, the matrix
  $M=\frac{1}{b}\ts B$ is Markov, and the final claim follows from the
  compatibility of the partial order on the positive cone with scaling
  by $b$ and the fact that the conditions in (3) and (4) are linear in
  $M$.
\end{proof}

\begin{example}\label{ex:mon-con}
  Let $M= (m^{}_{ij})^{}_{1\leqslant i,j \leqslant d}$ be Markov. When
  $d=2$, the non-negativity of $T^{-1} \nts M \ts T$ is equivalent to
  the single condition $\tr (M) \geqslant 1$, alternatively to
\[
     m^{}_{22} \, \geqslant \, m^{}_{12} \ts .
\]  
Likewise, for $d=3$, the original monotonicity condition for $M$, or
equivalently the non-negativity condition for $T^{-1} \nts M \ts T$,
boils down to
\[
   m^{}_{33} \, \geqslant \, m^{}_{23} \, \geqslant \, m^{}_{13}
   \quad \text{and} \quad 
   m^{}_{11} \, \geqslant \, m^{}_{21} \, \geqslant \, m^{}_{31} \ts ,
\]  
   where it was used that all rows of $M$ sum to $1$.
\exend
\end{example}

Let $E^{}_{\ell_1 \nts , \ldots, \ts \ell_d}$ denote the $\{ 0,1 \}$
Markov matrix with the row vector $e^{}_{\ell_i}$ as row $i$, for
$i \in [d\ts ]$, so
$E^{}_{\ell_1 \nts , \ldots, \ts \ell_d} = E^{}_{(1,\ts \ell_1)} +
\ldots + E^{}_{(d, \ts \ell_d)}$.  There exist $d^{\ts d}$ such
matrices, which are the extremal elements of $\cM_d$. Since
$e_i \preccurlyeq e_j$ if and only if $i\leqslant j$, it is clear that
$E^{}_{\ell_1 \nts , \ldots, \ts \ell_d}$ is monotone precisely when
$\ell_1 \leqslant \ell_2 \leqslant \dots \leqslant \ell_d$.  Using
\eqref{eq:elem-matrix}, or alternatively tracing the images of the
basis vectors $e_i$, one verifies the multiplication rule
\begin{equation}\label{eq:mult-E}
   E^{}_{k_1 \nts , \ldots, \ts k_d} \ts E^{}_{\ell_1 \nts , \ldots, \ts \ell_d}
   \, = \, E^{}_{\ell_{k_1} \nts \nts , \ldots, \ts \ell_{k_d}} .
\end{equation}
Due to the existence of singular idempotents among these matrices 
(when $d\geqslant 2$), one thus obtains the following 
{simple, but helpful} structure result.

\begin{fact}\label{fact:group}
  For\/ $d\geqslant 2$, the set of\/ $\{ 0, 1\}$ Markov matrices,
  under matrix multiplication, is a monoid, but not a group. The same
  property holds for the subset of monotone\/ $\{ 0, 1\}$ matrices.
  
  A\/ $\{ 0,1 \}$ Markov matrix in\/ $\cM_d$ is non-singular if and
  only if it is a permutation matrix. The subset of the\/ $d\ts !$
  permutation matrices is isomorphic with the symmetric group\/ $S_d$.
  \qed
\end{fact}

Now, we turn to {the convexity structure of $\mon$. While this
is certainly known, we are not aware of a source with a proof, whence
 we include one} for convenience.

\begin{lemma}\label{lem:convex}
  The set\/ $\mon$ is convex. It has\/ $\binom{2d-1}{d}$ extremal
  points, which are the monotone Markov matrices with entries in\/
  $\{ 0,1 \}$, that is, the\/
  $E^{}_{\ell_1 \nts , \ldots, \ts \ell_d}$ with\/
  $1\leqslant \ell_1 \nts \leqslant \ell_2 \leqslant \dots \leqslant
  \ell_d \leqslant d \ts$.
\end{lemma}

\begin{proof}
  Convexity is clear, for instance via
  Fact~\ref{fact:monotone}{\ts}(3).  Consequently, any convex
  combination of monotone $\{ 0,1 \}$ Markov matrices must lie in
  $\mon$. Thus, we first have to show that such convex combinations
  exhaust $\mon$. This follows from a greedy algorithm that is based
  on the following reduction argument.

  Consider a non-negative matrix $B\ne \nix$ with equal row sums, $b$
  say, where $b>0$. Assume that $B$ is monotone. Such a matrix, due to
  the monotonicity condition, appears in a (non-reduced) row-echelon
  form. This is captured in a set of integer pairs
  $\bigl( (i^{}_{1}, j^{}_{1}), \ldots , (i^{}_{r}, j^{}_{r}) \bigr)$,
  where $j^{}_{1}$ is the position of the first (or left-most)
  non-zero column of $B$, with $i^{}_{1}$ the lowest position of a
  positive element in it, $j^{}_{2}$ then is the left-most position of
  a column that is non-zero below row $i^{}_{1}$, with $i^{}_{2}$ the
  lowest position of a positive element in column $j^{}_{2}$, and so
  on.  Clearly, $r\geqslant 1$ and $i^{}_{r} = d$ due to $b>0$ in
  conjunction with $B$ being monotone.

For instance, $B$ may have the row-echelon form
\[
   B \, = \, 
      \left(\rule[0pt]{0pt}{48pt}
      \begin{array}{@{\,}cccccc@{\:}}
      0  & \multicolumn{1}{|c}{ *} & \cdot & \cdot & \cdot & \cdot \\
      0 & \multicolumn{1}{|c}{*}  & \cdot & \cdot & \cdot & \cdot \\
      0 & \multicolumn{1}{|c}{\bullet} & \cdot & \cdot & \cdot & \cdot\\
      \cline{2-2}
      0 & 0 &\multicolumn{1}{|c}{\bullet} & \cdot & \cdot & \cdot \\
      \cline{3-4}
      0 & 0 & 0 & 0 & \multicolumn{1}{|c}{*}  & \cdot \\
      0 & 0 & 0 & 0 & \multicolumn{1}{|c}{\bullet} & \cdot
      \end{array}\right)  , 
\]
here with $r=3$ and integer pairs $\bigl( (3,2), (4,3), (6,5)\bigr)$.
A symbol $\bullet$ marks the lowest positive element in a column, and
$*$ any element in the same column (above $\bullet$) that cannot be
smaller (as a consequence of monotonicity). In each row, there is thus
either one $*$ or one $\bullet$ by this rule. The total number of
symbols of type $\bullet$ or $*$ is $d$, so $6$ in this particular
case. To the left of them, all elements are $0$, while the remaining
elements of $B$ are left unspecified, as they play no role at this
stage.

Now, let $\alpha>0$ be the minimal element in the $\bullet$ positions,
which are the $(i^{}_{k}, j^{}_{k})$, and let $E$ denote the matrix
that has a $1$ in every $\bullet$ and in every $*$ position and a $0$
anywhere else, which obviously is a monotone $\{ 0,1 \}$ Markov
matrix, namely the $E^{}_{\ell_1 \nts , \ldots, \ts \ell_d}$ where
$\ell_i$ is the unique position of $*$ or $\bullet$ in row $i$, for
$i \in [d\ts ]$.  Now, set $B' = B - \alpha E$, which is still
monotone and has constant row sums
$b^{\ts\prime} = b -\alpha \geqslant 0$, but one $\bullet$ is now
replaced by a $0$, which means that this $\bullet$ is gone or has
moved up or right (or both) in the matrix. Unless $B'=\nix$, we repeat
the procedure with the new row-echelon form, which terminates after
finitely many steps. The result is a decomposition of $B$ as a sum of
monotone $\{0,1 \}$ Markov matrices with positive weight factors. If
we start with $M\nts\nts$, which has equal row sums $b=1$, it is clear
that we end up with a convex combination.

No monotone $\{0,1\}$ Markov matrix can be written as a convex
combination of the other ones, so their extremality is clear. Now,
given $d$, the monotone $\{ 0,1 \}$ Markov matrices are in obvious
bijection with the possibilities to distribute $d$ indistinguishable
balls (the $1\ts$s) to $d$ distinguishable boxes (the columns of the
matrix), where an outcome $(n^{}_{1}, \ldots , n^{}_{d})$, with
$n^{}_{1} \nts + \ldots + n^{}_{d} = d$, parameterises the matrix $M$ with
the row vector $e^{}_{1}$ in the first $n^{}_{1}$ rows, then
$e^{}_{2}$ in the next $n^{}_{2}$ rows, and so on.  The total number
of possible outcomes is well known to be
$\binom{\ts 2d\ts -1}{d\ts -1}= \binom{\ts 2 d\ts -1}{d}$, as this is
the number of choices to place $d-\nts 1$ separating walls between the
$d$ balls on the altogether $2d- \nts 1$ positions; see
\cite[A{\ts}001{\ts\ts}700]{OEIS} for details.
\end{proof}

\begin{table}
\begin{tabular}{c|cc|cc|c}
$(\ell^{}_{1}, \ell^{}_{2}, \ell^{}_{3}) \rule[-6pt]{0pt}{10pt}$ & 
    $M$ & $\sigma (M)$ & $p$ & $q$ & $M^2$ \\  \hline
$(1,1,1)\rule[-9pt]{0pt}{25pt} $ &
   $\left(\begin{smallmatrix} 1 & 0 & 0 \\ 1 & 0 & 0 \\ 1 & 0 & 0
            \end{smallmatrix}\right)$ & $\{ 1,0,0 \}$   &   
    $x^2 (x-1)$ & $x \ts (x-1)$ & $(1,1,1)$  \\
$(1,1,2)\rule[-9pt]{0pt}{24pt} $ &
   $\left(\begin{smallmatrix} 1 & 0 & 0 \\ 1 & 0 & 0 \\ 0 & 1 & 0
            \end{smallmatrix}\right)$ & $\{ 1,0,0 \}$   &       
    $x^2 (x-1)$ & $p(x)$ & $(1,1,1)$  \\
$(1,1,3)\rule[-9pt]{0pt}{24pt} $ &
   $\left(\begin{smallmatrix} 1 & 0 & 0 \\ 1 & 0 & 0 \\ 0 & 0 & 1
            \end{smallmatrix}\right)$ & $\{ 1,1,0 \}$    &     
    $x \ts (x-1)^2$ & $x \ts (x-1)$ & $(1,1,3)$  \\
$(1,2,2)\rule[-9pt]{0pt}{24pt} $ &
   $\left(\begin{smallmatrix} 1 & 0 & 0 \\ 0 & 1 & 0 \\ 0 & 1 & 0
            \end{smallmatrix}\right)$ & $\{ 1,1,0 \}$    &       
    $x \ts (x-1)^2$ & $x \ts (x-1)$ & $(1,2,2)$  \\
$(1,2,3)\rule[-9pt]{0pt}{24pt} $ &
   $\left(\begin{smallmatrix} 1 & 0 & 0 \\ 0 & 1 & 0 \\ 0 & 0 & 1
            \end{smallmatrix}\right)$ & $\{ 1,1,1 \}$     &     
   $(x-1)^3$ & $x-1$ & $(1,2,3)$  \\
$(1,3,3)\rule[-9pt]{0pt}{24pt} $ &
   $\left(\begin{smallmatrix} 1 & 0 & 0 \\ 0 & 0 & 1 \\ 0 & 0 & 1
            \end{smallmatrix}\right)$ & $\{ 1,1,0 \}$      &   
   $x \ts (x-1)^2$ & $x \ts (x-1)$ & $(1,3,3)$   \\
$(2,2,2)\rule[-9pt]{0pt}{24pt} $ &
   $\left(\begin{smallmatrix} 0 & 1 & 0 \\ 0 & 1 & 0 \\ 0 & 1 & 0
            \end{smallmatrix}\right)$ & $\{ 1,0,0 \}$       &     
   $x^2 (x-1)$ & $x \ts (x-1)$ & $(2,2,2)$  \\
$(2,2,3)\rule[-9pt]{0pt}{24pt} $ &
   $\left(\begin{smallmatrix} 0 & 1 & 0 \\ 0 & 1 & 0 \\ 0 & 0 & 1
            \end{smallmatrix}\right)$ & $\{ 1,1,0 \}$       &   
   $x \ts (x-1)^2$ & $x \ts (x-1)$ & $(2,2,3)$  \\
$(2,3,3)\rule[-9pt]{0pt}{24pt} $ &
   $\left(\begin{smallmatrix} 0 & 1 & 0 \\ 0 & 0 & 1 \\ 0 & 0 & 1
            \end{smallmatrix}\right)$ & $\{ 1,0,0 \}$        &  
   $x^2 (x-1)$ & $p(x)$ & $(3,3,3)$  \\
$(3,3,3)\rule[-9pt]{0pt}{24pt} $ &
   $\left(\begin{smallmatrix} 0 & 0 & 1 \\ 0 & 0 & 1 \\ 0 & 0 & 1
            \end{smallmatrix}\right)$ & $\{ 1,0,0 \}$     & 
    $x^2 (x-1)$ & $x \ts (x-1)$ & $(3,3,3)$  
\end{tabular}\bigskip

\caption{The $10$ extremal elements
  $E^{}_{\ell_1 \nts , \ts \ell_2 , \ts \ell_3}$ of
  $\cM^{}_{3,\preccurlyeq}$ with some of their properties. Here,
  $\sigma(M)$ is the spectrum of $M$ with multiplicities, while $p$
  and $q$ are, up to an overall sign, the characteristic and the
  minimal polynomial of $M$. Note that $p\ne q$ precisely when $M$ is
  an idempotent. The last column gives $M^2$ in terms of its index
  parameters.}
\label{tab:extremal}   
\end{table}

The case $d=3$ is summarised in Table~\ref{tab:extremal}.  By
Lemma~\ref{lem:convex}, every monotone Markov matrix $M\in\cM_d$ can
be expressed as a convex combination of the form
\begin{equation}\label{eq:mon-conv}
    M \, = \sum_{1\leqslant  \ell_{1} \leqslant \cdots
       \leqslant \ell_{\nts d} \leqslant d }
     \! \alpha^{}_{\ell_{1} \nts , \ldots, \ts \ell_{d}} \,
     E^{}_{\ell_{1} \nts , \ldots, \ts \ell_{d}} 
\end{equation}
with all coefficients
$\alpha^{}_{\ell_{1} \nts , \ldots, \ts \ell_{d}}\geqslant 0$, their
sum being $1$, and $E^{}_{\ell_{1} \nts , \ldots, \ts \ell_{d}}$ as
above.  Observing that
$\ts \tr \bigl( E^{}_{\ell_{1} \nts , \ldots, \ts \ell_{d}} \bigr)
\geqslant 1$ whenever
$\ell_1 \nts \leqslant \ell_2 \leqslant \dots \leqslant \ell_d$, one
finds from \eqref{eq:mon-conv} that $\tr (M) \geqslant 1$ holds for
all monotone Markov matrices, which easily generalises as follows.

\begin{coro}\label{coro:trace}
  Let\/ $B \in \Mat (d,\RR)$ be a non-negative matrix with equal row
  sums, $b\geqslant 0$.  If\/ $B$ is also monotone, one has\/
  $\tr (B) \geqslant b$.  \qed
\end{coro}

\subsection{Monotonicity and embeddability}

{As in \cite{BS}, we use $\cE_d$ to denote} the semigroup generated by
the embeddable Markov matrices of dimension $d$.  For $d=2$, every
element of $\cE_2$ is itself embeddable (so
$\cE^{}_{2} = \cM^{\mathrm{E}}_{2}$, which is no longer true for
$d\geqslant 3$), and the set of monotone Markov matrices agrees with
the closure of $\cE_2$. In fact, $\cM_{2,\preccurlyeq}$ is the closed
triangle in $\cM_2$ with the vertices $\one^{}_{2}$,
$E^{}_1=\left( \begin{smallmatrix} 1 & 0 \\ 1 &
    0 \end{smallmatrix}\right)$ and
$E^{}_2=\left( \begin{smallmatrix} 0 & 1 \\ 0 &
    1 \end{smallmatrix}\right)$.  Only the line
$\{ \alpha E^{}_1 + (1-\alpha) E^{}_2 : 0 \leqslant \alpha \leqslant 1
\}$ does not belong to $\cE_2$, because it consists of singular
idempotents. The only other idempotent in $\cM^{}_2$ is {$\one$,} the
trivial case. This leads to the following result.

\begin{prop}\label{prop:mon-2}
  An element\/ $M\in \cM^{}_2$ is monotone if and only if\/
  $\tr (M) \geqslant 1$. Thus, being monotone is equivalent to either
  being embeddable or being a non-trivial idempotent.
\end{prop}

\begin{proof}
  Observe that $\tr (M) \leqslant 2$ holds for all $M\in \cM^{}_2$.
  Since
  $M=\left( \begin{smallmatrix} 1-a & a \\ b & 1-b \end{smallmatrix}
  \right)$ with\/ $a,b\in [ \ts 0,1]$ is monotone if and only if
  $1\geqslant a+b$, compare Example~\ref{ex:mon-con}, the first claim
  is immediate. By Lemma~\ref{lem:Kendall}, $\tr (M) = 1$ means $M$ is
  monotone, but not embeddable.

  For the second claim, {recall} that $M\in \cM^{}_2$ is an idempotent
  if and only if $M^2=M$, which implies
  $\sigma (M) \subseteq \{ 0,1\}$.  Since $\lambda = 1$ is always an
  eigenvalue, being an idempotent either means that the second
  eigenvalue is also $1$, hence $M=\one$ by Lemma~\ref{lem:idem}, or
  that $\det (M)=0$, which gives the line from $E^{}_1$ to $E^{}_2$
  discussed above.
\end{proof}

\begin{coro}\label{coro:div-2}
  Any\/ $M\in\cM^{}_{2,\preccurlyeq}$ is infinitely divisible within\/
  $\cM^{}_{2,\preccurlyeq}$. In fact, $M\in\cM^{}_{2}$ is infinitely
  divisible if and only if it is monotone.
\end{coro}

\begin{proof}
  Let $M \in \cM^{}_2$ be monotone.  By Proposition~\ref{prop:mon-2},
  the case $\det (M) = 0$ means $M^2=M$, so also $M^n=M$ for all
  $n\in\NN$ by induction, and $M$ is a monotone $n$-th root of itself.
  Clearly, the latter statement also applies to $M=\one$.
   
When $M=\left(\begin{smallmatrix} 1-a & a \\ b & 1-b
\end{smallmatrix}\right) \ne \one$ is embeddable, 
we have $a+b>0$ and $M=\ee^Q$ with the unique generator $Q$ from
Lemma~\ref{lem:Kendall}. Then, for any $n\in\NN$, a Markov $n$-th
root of $M$ is given by
\[
   \exp \bigl(\tfrac{1}{n}\ts Q \bigr) \, = \, 
   \begin{pmatrix} 1\nts -\epsilon \ts a & \epsilon \ts a \\
   \epsilon \ts b & 1\nts - \epsilon \ts b \end{pmatrix}
   \quad\text{with } \, \epsilon \, = \, 
   \frac{1 - \! \sqrt[n\,]{1-a-b\ts\ts}}{a+b}\ts ,
\]   
as follows from the same standard calculation with the matrix
exponential that was used to derive \eqref{eq:ei-gen}.  Now,
$\exp \bigl(\frac{1}{n}\ts Q \bigr)$ is monotone if and only if
$1 \geqslant \epsilon (a+b)$, by an application of the criterion from
Example~\ref{ex:mon-con}. But this estimate follows from $0<a+b < 1$
because $\epsilon \in [ \ts 0, 1]$.

It remains to show that infinite divisibility of $M\in\cM^{}_{2}$
implies its monotonicity, which can be derived from the spectrum as
follows. If $1$ is the only eigenvalue of $M$, we have $M=\one$ by
Lemma~\ref{lem:idem}, which is embeddable.  Otherwise, one has
$\sigma (M) = \{1,\lambda\}$ where $\lambda \ne 1$ must be real, with
$\lvert \lambda \rvert \leqslant 1$. Since $M$ has a Markov square
root by assumption, we get $\lambda \in [ \ts 0,1)$. Now,
$\lambda = 0$ means that $M$ is an idempotent, while $\lambda>0$
implies $\det (M) > 0$, so $M$ is embeddable by
Lemma~\ref{lem:Kendall}.  Monotonicity of $M$ now follows from
Proposition~\ref{prop:mon-2}.
\end{proof}

Our next goal is a better understanding of the connection between
$\mon$ and $\cC_d$, aiming at generalisations of
Corollary~\ref{coro:div-2} to general $d$. To this end, we once more
consider a non-negative matrix $C$ with equal rows and parameter sum
$c$, with $C=\nix$ only when $c=0$. For $c>0$, the matrix
$\frac{1}{c} \ts C$ is both Markov and monotone. Consequently, the
Markov matrix $M^{}_{C}$ from Eq.~\eqref{eq:equal-input}, for any
$c\in [ \ts 0,1]$, is a convex combination of $\one$ and
$\frac{1}{c} \ts C$, hence monotone as well.

When $c=0$, which means $M^{}_{C}=\one$, or when $c=1$, where
$M^{}_{C} = C$, the matrix $M^{}_{C}$ is a monotone idempotent. When
$c\in (0,1)$, Eq.~\eqref{eq:ei-gen} implies
\begin{equation}\label{eq:gen-gen}
    M^{}_{C} \, = \, \ee^Q  \quad\text{with }\,
    Q \, = \, - \frac{\log (1-c)}{c} \ts Q^{}_{C} \ts ,
\end{equation}
where $Q^{}_{C} = M^{}_{C} - \one$ as before, which is an equal-input
generator. Now, for arbitrary $n\in\NN$, a standard calculation with
the exponential series gives the formula
\begin{equation}\label{eq:n-roots}
  \exp \bigl( \tfrac{1}{n} \ts Q \bigr) \, = \,
  \one + \frac{1 - \! \sqrt[n\,]{1-c \ts}}{c} \ts Q^{}_{C}
  \, = \, \sqrt[n\,]{1-c \ts} \,\ts \one +
  \bigl( 1 - \! \sqrt[n\,]{1-c \ts } \, \bigr)  \tfrac{1}{c} \ts C \ts .
\end{equation}
Since $\sqrt[n\,]{1-c \ts } \in (0,1)$ under our assumptions, this is
a convex combination of two monotone Markov matrices, hence monotone
and Markov itself. We have thus proved the following generalisation of
our previous statements, assuming $d\geqslant 2$ as usual. In
particular, the idempotent elements play a similar role as in the
two-dimensional case.

\begin{theorem}\label{thm:monotone-2}
  Let\/ $C$, $Q^{}_{C}$ and\/ $M^{}_{C}$ be as above, and let\/
  $c=c^{}_{1} \nts + \ldots + c^{}_{\nts d}$ be the corresponding parameter
  sum. If\/ $c\in [\ts 0,1]$, $M^{}_{C}$ is Markov and monotone, with
  the following properties.
\begin{enumerate}\itemsep=2pt
\item $M^{}_{C}$ is an idempotent if and only if\/ $c \in \{ 0,1 \}$,
  where\/ $c=0$ means\/ $M^{}_{C} = \one$.
\item $M^{}_{C}$ is embeddable if and only if\/ $c\in [ \ts 0,1)$,
   then with\/ $Q=\nix$ for\/ $c=0$ or otherwise with the 
   generator\/ $Q$ from\/ \eqref{eq:gen-gen}.
\end{enumerate}
In particular, $M^{}_{C}=\one$ is the only embeddable idempotent.  
{Further, for all\/ $c\in [0,1]$ and} for
every\/ $n\in\NN$, $M^{}_{C}$ has a Markov\/ $n$-th
root that is both equal-input and monotone, which is to say that\/
$M^{}_{C}$ is infinitely divisible within\/ $\mon \nts \cap \cC_d$.
\qed
\end{theorem}

Note that, also in generalisation of the case $d=2$, the set of
monotone Markov matrices of type $M^{}_{C}$ with summatory parameter
$c\in [ \ts 0,1]$ is the closure of the equal-input Markov matrices
that are embeddable with an equal-input generator, with all
non-embeddable boundary cases being non-trivial idempotents. In fact,
one has more as follows.

\begin{coro}\label{coro:convex-2}
  A Markov matrix\/ $M\in \cC_d$ is monotone if and only if its
  summatory parameter satisfies\/ $c \in [ \ts 0,1]$. So, one obtains
  the convex set
\[
    \cmon \, \defeq \, \cC_d \cap \mon \, = \, 
    \big\{ M\in \cC_d : c \in [ \ts 0,1] \big\} ,
\]   
  with the\/ $d+\nts 1$ extremal elements\/ 
  $E^{}_{1}, \ldots , E^{}_{d}$ and\/ $\one$.
  
  Further, $\cmon$ is the disjoint union of the set of equal-input
  embeddable elements from\/ $\cC_d$ with the set\/ $\cC_{d,1}$ of
  non-trivial idempotents in\/ $\cC_d$. The eigenvalues of any\/
  $M\in\cmon$ are real and non-negative, and they are positive
  precisely for the embeddable cases.
\end{coro}

\begin{proof}
  It is clear from Theorem~\ref{thm:monotone-2} that all $M\in \cC_d$
  with $c\in [ \ts 0,1]$ are monotone, so we need to show that no
  further element of $\cC_d$ is. To this end, consider a Markov matrix
  of the form $M = (1-c) \one + C$ with $c>1$, which implies that all
  $c^{}_{i} > 0$. Then, it is easy to check that $T^{-1} M \ts T$
  fails to be a non-negative matrix, where $T$ is the matrix from
  \eqref{eq:T-def}, and $M$ fails to be monotone by
  Fact~\ref{fact:monotone}{\ts}(3).
  
  When $c\in [\ts 0,1]$, we have
  $0 \leqslant c^{}_{i} \leqslant c^{}_{1} \nts + \ldots + c^{}_{d} = c \ts
  \leqslant 1$ for all $1 \leqslant i \leqslant d$, and
\[
    M^{}_{C} \, = \, (1-c)\ts \one + C \, = \, (1-c)\ts \one +
    \sum_{i=1}^{d} c^{}_{i} \, E^{}_{i}
\]  
  is a convex combination, where the extremality of
  $E^{}_{1}, \ldots , E^{}_{d}$ and\/ $\one$ is clear.
  
  Another application of Theorem~\ref{thm:monotone-2} gives the
  decomposition claimed, while the statement on the spectrum is clear
  because the eigenvalues of $M^{}_{C}$ are $1$ and $1-c$.
\end{proof}
 
Geometrically, the situation is that the simplex $\cmon$ separates the
compact set $\cC_d$ into the subset with $c\in [ \ts 0,1)$, which are
the `good' cases for embeddability, and the subset with
$c\in \bigl( 1, \frac{d}{d-1}\bigr]$, where embeddability requires $d$
even and further conditions, but is never possible with an equal-input
generator. For $d=2$, we refer to \cite[Fig.~1]{BS} for an
illustration. \smallskip

One can view $\cmon$ differently when starting in $\mon$. Let $S_d$ be
the symmetric (or permutation) group of $d$ elements, and $P_{\pi}$
for $\pi\in S_d$ the standard permutation matrix that represents the
mapping $e^{}_{i} \mapsto e^{}_{\pi (i)}$ under multiplication to the
right. $P_{\pi}$ has elements $\delta^{}_{i, \pi (j)}$ and satisfies
$P^{-1}_{\pi} = P^{}_{\pi^{-1}}$.  There are $d\ts !$ such matrices,
the extremal elements among the doubly stochastic matrices mentioned
earlier. The conjugation action by such a matrix gives
\[
    P^{}_{\pi} \ts E^{}_{( k,\ell )} \ts  P^{-1}_{\pi} \, = \,
    E^{}_{\left( \pi (k), \pi (\ell) \right)}
    \quad \text{and} \quad
    P^{}_{\pi} \ts E^{}_{\ell_1 \nts , \ldots , \ts \ell_d}
    \ts  P^{-1}_{\pi} \, = \,
    E^{}_{\pi (\ell_{\pi^{-1}(1)}), \ldots , \pi (\ell_{\pi^{-1} (d)})} \ts ,
\]
as follows from a simple calculation with
$E^{}_{\ell_1 \nts , \ldots , \ts \ell_d} = E^{}_{(1,\ts \ell_1)} +
\ldots + E^{}_{(d, \ts \ell_d)}$, or, alternatively, from tracing the
images of the basis vectors $e_i$ for $1 \leqslant i \leqslant d$.

Now, a set $\cF\subseteq\cM_d$ of Markov matrices is called
\emph{permutation invariant} if
$P^{}_{\pi} \ts\ts \cF P^{-1}_{\pi} = \cF$ holds for all
$\pi \in S_d$. Clearly, $\cM_d$ itself is such a set, as is $\cC_d$ or
its subset of constant-input matrices.  The latter are also
\emph{individually} permutation invariant (which is also called
\emph{exchangeable} in probability theory \cite{Feller}), that is,
$P^{}_{\pi} \ts M P^{-1}_{\pi} = M$ holds for every constant-input
matrix $M$ and all $\pi \in S_d$. In fact, the Markov matrices that
are individually permutation invariant are precisely the
constant-input ones, without restriction on the summatory parameter
$c$.

The set of all $\{ 0, 1 \}$ Markov matrices is permutation invariant
as well, and partitions into $S^{}_d\ts $-orbits of the form
$\cO^{}_{\nts S_d} (M) = \{ P^{}_{\pi} \ts M P^{-1}_{\pi} : \pi \in
S^{}_d \}$. Two such orbits are
\[
    \cO^{}_{\nts S_d} (\one) \, = \, \{ \one \}
    \quad \text{and} \quad
    \cO^{}_{\nts S_d} (E^{}_{1}) \, = \,
    \{ E^{}_{1}, \ldots , E^{}_{d} \} \ts ,
\]
which both consist of monotone matrices only. One can check that no
other orbit of the decomposition has this property, which implies the
following characterisation.

\begin{fact}
  The convex set\/ $\cmon$ is the maximal subset of\/ $\mon$ that is
  permutation invariant. The elements of\/ $\cmon$ that are\/
  \emph{individually} permutation invariant are the constant-input
  matrices with\/ $c\in [\ts 0,1]$.  \qed
\end{fact}

Let us turn to Markov semigroups.  Recall from \cite{KK} that a
(homogeneous) Markov semigroup $\{ \ee^{t Q} : t \geqslant 0\}$, with
generator $Q$, is called \emph{monotone} when $\ee^{t Q}$ is monotone
for every $t\geqslant 0$.  Moreover, a generator $Q$ is called
\emph{monotone} if all off-diagonal elements of $T^{-1}\nts Q \ts T$
are non-negative, where $T$ is the matrix from Eq.~\eqref{eq:T-def}.
This concept is motivated by the following connection, which is
{a minor variant of} \cite[Thm.~2.1]{KK}. Due to its 
importance, we include a short proof {that is tailored to
our later needs in this context.}

\begin{prop}\label{prop:mon-gen}
  If\/ $Q$ is a Markov generator, the following properties are
  equivalent.
\begin{enumerate}\itemsep=2pt
\item The semigroup\/ $\{ \ee^{t Q} : t \geqslant 0\}$ is monotone.
\item The generator\/ $Q$ is monotone.
\end{enumerate}  
\end{prop}

\begin{proof}
  For $(1) \Rightarrow (2)$, observe that
  $T^{-1} \ee^{t Q} \ts\ts T \geqslant \nix$ implies
  $\bigl( T^{-1} \frac{1}{t} (\ee^{tQ} - \one) \ts T \bigr)_{ij}
  \geqslant 0$ for $t>0$ and all $i\ne j$. Then, taking
  $t \, \raisebox{1.5pt}{$\scriptstyle{\searrow}$} \, 0$ establishes
  this direction.
  
  For $(2) \Rightarrow (1)$, is it clear that
  $T^{-1} \nts Q \ts\ts T \nts + \alpha \ts \one \geqslant \nix$ holds
  for any sufficiently large $\alpha > 0$. Choose $\alpha$ also large
  enough so that $M^{}_{\alpha} \defeq \one + \alpha^{-1} Q$ is
  Markov, which is clearly possible. Then,
  $T^{-1} \nts M^{}_{\alpha} \ts T \geqslant \nix$, and we get
  $T^{-1} \nts M^{m}_{\alpha} \ts T = (T^{-1} \nts M^{}_{\alpha} \ts T
  )^m \geqslant \nix$ for all integers $m\geqslant 0$. Now, observe
\[
    \ee^{t Q} \ts = \, \ee^{-\alpha t} \ee^{\alpha t M_{\alpha}} \ts =
    \sum_{m=0}^{\infty} \ee^{-\alpha t}\ts \frac{(\alpha t)^m}{m !}
    \, M^{m}_{\alpha}  ,
\]  
which, for all $t\geqslant 0$, constitutes a convergent sum that is a
convex combination of monotone Markov matrices. Consequently,
$\ee^{t Q}$ is monotone as well.
\end{proof}

\begin{example}\label{ex:gen-con}
  Let $Q = (q^{}_{ij})^{}_{1\leqslant i,j \leqslant d}$ be a Markov
  generator.  When $d=2$, it is always monotone, that is, no extra
  condition emerges; compare Proposition~\ref{prop:mon-2}.  When
  $d=3$, being monotone is equivalent with the two conditions
\[
     q^{}_{23} \, \geqslant \, q^{}_{13} 
     \quad \text{and} \quad
     q^{}_{21} \, \geqslant \, q^{}_{31} \ts ,
\]   
which provide a surprisingly simple criterion for monotonicity
in this case.
\exend
\end{example}

The above considerations have the following consequence.

\begin{coro}
   For\/ $M \in \mon$, the following properties are equivalent.
\begin{enumerate}\itemsep=2pt
\item $M$ is embeddable via a monotone Markov generator.
\item $M$ is non-singular and infinitely divisible within\/ $\mon$.
\end{enumerate}
\end{coro}

\begin{proof}
  (1) $\Rightarrow$ (2): Let $M=\ee^Q$ with $Q$ a monotone generator,
  so $\det (M) = \ee^{\tr (Q)} >0$. Now, $\frac{1}{n}\ts Q$ is still a
  monotone generator, for any $n\in\NN$, and
  $\exp \bigl( \frac{1}{n}\ts Q\bigr)$ is a monotone Markov matrix (by
  Proposition~\ref{prop:mon-gen}) that is also an $n$-th root of
  $M \nts $.

  (2) $\Rightarrow$ (1): By Kingman's characterisation
  \cite[Prop.~7]{King}, $M$ is embeddable, so $M=\ee^Q$ with some
  generator $Q$. We need to show that $Q$ can be chosen to be
  monotone. Let $R_n$ be an $n$-th root of $M$ that is Markov and
  monotone, which exists and implies that
  $A_n \defeq n\ts ( R_n \nts - \one )$ is a monotone generator. By a
  standard compactness argument, there is a subsequence
  $(n^{}_{i})^{}_{i\in\NN}$ of increasing integers such that
  $Q{\ts}' = \lim_{i\to\infty} A_{n_i}$ is a monotone generator as
  well.  From here on, we can employ Kingman's original proof to
  conclude that
\[
  M \, = \, \Bigl( \one + \myfrac{Q{\ts}'}{n^{}_{i}} \Bigr)^{n^{}_{i}}
  + \, o (1) \qquad \text{as } i\to\infty \ts ,
\]
which gives $M = \ee^{Q'}$ as claimed.
\end{proof}

\subsection{Idempotents and infinite divisibility}

At this point, it seems worthwhile to take a closer look at infinite
divisibility in general. In this context, we refer to
\cite[Sec.~X.9]{Feller} for the underlying (pseudo{\ts}-)Poissonian
structures.

\begin{prop}\label{prop:Pois}
  Let\/ $P^{}_{0} , P \in \cM^{}_{d}$ be chosen such that\/
  $P^{2}_{0} = P^{}_{0}$, so\/ $P^{}_{0}$ is an idempotent, and that\/
  $P^{}_{0} P = P P^{}_{0} = P$. Then, the matrix family\/
  $\{ M(t) : t \geqslant 0 \}$ with
\[
     M (t) \, \defeq \, \ee^{-t} \Bigl( P^{}_{0} + \sum_{m=1}^{\infty}
     \myfrac{t^m}{m!} \, P^m \Bigr) \, = \:
     \ee^{-t} \bigl( P^{}_{0} - \one + \ee^{t P} \bigr)
\]    
   satisfies the following properties, where\/ 
   $A \defeq P\nts - \one$ is a {Markov} generator.
\begin{enumerate}\itemsep=2pt
\item The mapping\/ $t\mapsto M(t)$ is continuous, with\/
  $M (0) = P^{}_{0}$.
\item $M(t) \ts M(s) = M(t+s)$ holds for all\/ $t,s \geqslant 0$.
\item $M(t)$ is Markov, for all\/ $t \geqslant 0$. 
\item $M(t) = \ee^{-t} ( P^{}_{0} \nts - \one ) +
   \ee^{t A} = P^{}_{0} \ts \ee^{tA}$ for all\/ $t\geqslant 0$.
\item $P=P^{}_{0}$ if and only if\/ $M(t) = P^{}_{0}$ holds for 
   all\/ $t\geqslant 0$.   
\item For\/ $t\geqslant 0$, $M (t)$ is idempotent if and only if\/
   $M (t) = P^{}_{0}$.
\end{enumerate}   
In particular, $\{ M(t) : t \geqslant 0 \}$ always constitutes a
continuous monoid, with\/ $P^{}_{0}$ as its neutral element, while it
is a homogeneous Markov semigroup if and only if\/ $P^{}_{0} = \one.$
\end{prop}

\begin{proof}
  (1) is obvious, while (2) follows from a standard calculation with
  the convergent series. Both $P^{}_{0}$ and $P$ are Markov, and so is
  $P^m$ for all $m\in\NN$. Now, for any $t\geqslant 0$, $M(t)$ is a
  con\-vergent, convex combination of Markov matrices, hence Markov as
  well, which shows (3).
  
  Next, claim (4) and the easy direction of (5) follow from elementary
  calculations, using $P^{}_{0} \ts \ee^{tP} = P^{}_{0} +
  \sum_{m\geqslant 1} \frac{t^m}{m !} P^m$ together with  
  $\ee^{t P} = \ee^t \ee^{t A}\nts$. When $M (t) = P^{}_{0}$ for all
  $t\geqslant 0$, one obtains $P^{}_{0} + \ee^{tP} = \one + \ee^t
  P^{}_{0}$. But, {observing $\ee^{t P^{}_0} = \one - P^{}_{0}
  + \ee^t P_0$,} 
  this implies $\ee^{t P} = \ee^{t P^{}_{0}}$ for all
  $t\geqslant 0$ and thus $P=P^{}_{0}$.

  For (6), one direction is trivial.  The other follows from (4) by
  observing that $M (t)$ idempotent implies
  $P^{}_{0} \ts \ee^{tA} = M (t) = M (t)^2 = M(2\ts t) = P^{}_{0}\ts
  \ee^{2 \ts t A}$ and hence gives the relation
  $P^{}_{0} = P^{}_{0} \ts \ee^{tA} = M(t)$ as claimed.
  
  Finally, while the (abstract) semigroup and monoid properties are
  clear, the family can only be a homogeneous Markov semigroup when
  $M(t) = \ee^{t Q}$ for some generator $Q$ and all $t\geqslant 0$, so
  $M(0)=\one$, which is the only non-singular idempotent in $\cM_d$ by
  Lemma~\ref{lem:idem}. We thus have $P^{}_{0} = \one$ in this case,
  and (4) implies $M(t) = \ee^{tA}$ as claimed.
\end{proof}

\begin{example}\label{ex:Pois}
  Let us analyse the meaning of Proposition~\ref{prop:Pois} for
  $d\geqslant 2$.  If $P^{}_{0}=\one$, there is no further restriction
  on $P$, and $M(t) = \ee^{t A}$ with $A=P\nts -\one$, which can thus
  be any generator with diagonal elements in $[-1,0]$. This way,
  possibly after rescaling $t$, all embeddable matrices are covered.
  For $P\in\cC_d$ with $c\in [\ts 0,1)$, we see that $A=P-\one$ is a
  generator of equal-input type, in line with Lemma~\ref{lem:Kendall}
  and Proposition~\ref{prop:ei-embed}.
  
  If $P^{}_{0} \in \cC^{}_{d,1}$, we get
  $P^{}_{0} = \sum_{i=1}^{d} \beta_i E_i$ with $\beta_i \geqslant 0$
  and $\beta^{}_{1} \nts + \ldots + \beta^{}_{d} = 1$, hence
  $P^{}_{0} = P \ts P^{}_{0}$ for any $P\in\cM_d$.  Assuming
  $P P^{}_{0} = P^{}_{0} P = P$, we find $P = P^{}_{0}$, which gives
  $M(t) \equiv P^{}_{0}$ by Proposition~\ref{prop:Pois}{\ts}(5).
  
  For $d=2$, where $\cM^{}_{2} = \cC^{}_{2}$, this exhausts all cases
  because no further idempotents exist. Consequently, $M\in\cM_2$ is
  infinitely divisible if and only if it is embeddable or an
  idempotent, with $M\nts = \one$ being the only case that is both;
  compare Proposition~\ref{prop:mon-2} and Corollary~\ref{coro:div-2}.

  When $d\geqslant 3$, one obtains mixtures via direct sums, where
  $\one \oplus P^{}_{0}$ and $(\one \nts + \nts \nts A) \oplus P^{}_{0}$ 
  lead to $M(t) = \ee^{t A} \oplus P^{}_{0}$. There are further examples
  for $d=3$, compare Example~\ref{ex:idem},  such as
\[
  P^{}_{0} \, = \, \alpha E^{}_{1,1,3} + (1\nts -\alpha) E^{}_{1,3,3}
  \, = \, \begin{pmatrix} 1 & 0 & 0 \\
   \alpha & 0 & 1\! - \nts\nts \alpha \\
    0 & 0 & 1 \end{pmatrix} ,  \quad
    \text{with } \det (P^{}_{0}) = 0 \ts ,
\]  
which is idempotent for any $\alpha \in [\ts 0,1]$. Then,
$P\in\cM^{}_{3}$ with $P \ts P^{}_{0} = P^{}_{0} \ts P = P$ leads to
\[
  P \, = \, \begin{pmatrix} a & 0 & 1\! - \nts\nts a \\
    c & 0 & 1\! - \nts\nts c \\ 1\! - \nts\nts b & 0 & b \end{pmatrix}
\] 
with $a,b \in [\ts 0,1]$ and
$ c = \alpha \ts a + (1\nts -\alpha) (1\nts -b)$.  Here,
$M (t) = P^{}_{0} \ts \ee^{tA}$ is singular for all $t\geqslant 0$ and
thus never embeddable. Moreover, for $P\ne P^{}_{0}$, the matrix
$M (t)$ can only be idempotent for $t=0$ and possibly for
\emph{isolated} further values of $t>0$. This is so because
$M (t) = M (t)^2 = M (2 \ts t)$ implies
$\ee^{t P} = \ee^{ t P^{}_{0}}$ via an elementary calculation.  When
this holds for $t$ from a set with an accumulation point, $t^{}_0$
say, standard arguments imply $P=P^{}_{0}$.  So, more interesting as
well as more complicated cases emerge in
$\cM_d \nts \setminus \nts \cC_d$ as $d$ grows.  \exend
\end{example}

\begin{lemma}\label{lem:dicho}
  With\/ $P^{}_{0}$, $P$ and\/ $M (t)$ as in
  Proposition~\textnormal{\ref{prop:Pois}}, one has either\/
  $\ts\det \bigl( M (t) \bigr) > 0$ for all\/ $t\geqslant 0$, which
  happens if and only if\/ $P^{}_{0} = \one$, or\/
  $\ts\det \bigl( M (t)\bigr) = 0$ for all\/ $t\geqslant 0$, which is
  true whenever\/ $P^{}_{0} \ne \one$, equivalently whenever\/
  $P^{}_{0}$ is singular.
\end{lemma}

\begin{proof}
  Recall via Lemma~\ref{lem:idem} that $P^{2}_{0} = P^{}_{0}$ either
  means $\det (P^{}_{0}) = 1$, which forces $P^{}_{0} = \one$, or
  $\det (P^{}_{0}) = 0$. Now, observe that
  $P^{}_{0} \ts\ts M(t) = M(t)$ holds for all $t\geqslant 0$, so
\[
    \det (P^{}_{0}) \det \bigl( M(t)\bigr)
    \, = \: \det \bigl( M(t) \bigr),
\]  
and $\det \bigl( M (t)\bigr) \equiv 0$ for singular $P^{}_{0}$ is
immediate.
  
The only remaining case is $P^{}_{0} = \one$.  Here,
Proposition~\ref{prop:Pois}{\ts}(3) implies $M(t) = \ee^{tA}$ with
$A=P\nts-\one$ and hence
$\det \bigl( M(t)\bigr) = \ee^{\tr ( t A)} > 0$.
\end{proof}

A semigroup as in Proposition~\ref{prop:Pois} is called
\emph{Poissonian} if $P^{}_{0} = \one$, and
\emph{pseudo{\ts}-Poissonian} otherwise \cite[Sec.~X.1]{Feller}. We
can now recall the central classification result on infinitely
divisible, finite-dimensional Markov matrices from \cite{GMMS} as
follows. {It seems a bit hidden in the literature, but nicely} 
underpins the role of idempotents in the embedding problem.

\begin{theorem}\label{thm:divisible}
  A Markov matrix\/ $M\in\cM_d$ is infinitely divisible if and only if
  there are Markov matrices\/ $P^{}_{0}, P \in \cM^{}_{d}$, with\/
  $P^{2}_{0} = P^{}_{0}$ and\/ $P^{}_{0} P = P P^{}_{0} = P$, and
  some\/ $s\geqslant 0$ such that
\[
    M \, = \, \ee^{-s} \Bigl( P^{}_{0} + \sum_{m=1}^{\infty}
     \myfrac{s^m}{m!} \ts P^m \Bigr) .
\]  
  Moreover, $M$ is embeddable if and only
  if one also has\/ $P^{}_{0} = \one$.  
\end{theorem}

\begin{proof}
  For the proof of the first claim, we refer to \cite{GMMS}. 
  
  For the second claim, we know that $M$ embeddable implies
  $\det (M) > 0$, and we are in the case with $P^{}_{0} = \one$ by
  Lemma~\ref{lem:dicho}. Conversely, when $P^{}_{0} = \one$,
  Proposition~\ref{prop:Pois}{\ts}(4) gives $M = \ee^{sA}$ with the
  generator $A=P\nts -\one$.
\end{proof}

Note that the parameter $s$ cannot be avoided in this formulation
because a generator $Q$ can have {diagonal entries of
arbitrarily large negative value,}
whence $\one + Q$ need not be Markov, while $\one + s \ts Q$, for all
suitably small $s>0$, will be; compare Example~\ref{ex:Pois} and
Proposition~\ref{prop:Pois}.

Further consequences can be derived from $[P^{}_{0}, P \ts ]=\nix$
when $P$ is \emph{cyclic}, which means that minimal and characteristic
polynomial of $P$ agree. In particular, this is the case when $P$ is
simple; see \cite[Fact~2.10]{BS} for a systematic characterisation of
cyclic matrices. Whenever $P\in\cM_d$ is cyclic, its
\emph{centraliser} is the Abelian ring
\[
    \Co (P) \, = \, \{ B \in \Mat (d, \RR) :
      [ P , B \ts ] = \nix \}  \, = \, \RR [P \ts ] \ts ,
\]
where each element of this ring is of the form
$\sum_{n=0}^{d-1} \alpha^{}_{n} P^n$ with all $\alpha^{}_{n} \in \RR$,
as a consequence of the Cayley--Hamilton theorem.  In particular,
$P^{}_{0}$ is then an idempotent from this ring. We leave further
details to the interested reader.

\section{Monotone Markov matrices in three
    dimensions}\label{sec:three}

  Let us now look at $d=3$ in more detail, where we state the
  following simple and certainly well-known property, which we 
  {also} prove due to lack of reference.
  
\begin{prop}\label{prop:mon-3}
  The eigenvalues of any\/ $M\in\cM^{}_{3,\preccurlyeq}$ are real.
  Moreover, at most one eigenvalue of\/ $M$ can be negative, which
  happens if and only if\/ $\det (M) < 0$.
   
  Further, if\/ $d=3$ and\/ $Q$ is a monotone Markov generator, its
  eigenvalues are non-positive, real numbers.
\end{prop}

\begin{proof}
  First, $M\in\cM^{}_{3,\preccurlyeq}\subset \cM^{}_{3}$ implies
  $1\in\sigma (M)$. As $d=3$, the characteristic poly\-nomial of 
  $M$ then is
  $(1-x) \bigl( x^2 - (\tr (M) - 1) x + \det (M) \bigr)$, and the
  remaining two eigenvalues are
\begin{equation}\label{eq:eig}
  \lambda^{}_{\pm} \, = \, \myfrac{1}{2}
  \bigl( \tr (M) - 1 \pm \sqrt{\Delta} \, \bigr) ,
\end{equation}
with the discriminant $\Delta = ( \tr (M) - 1 )^2 - 4 \ts \det
(M)$. By an explicit calculation, where one first eliminates
$m^{}_{22}$ and later also $m^{}_{11}$ via the row sum condition, one
verifies that
\[
  \Delta \, = \,
   \bigl( m^{}_{11} - m^{}_{21} + m^{}_{23} - m^{}_{33} \bigr)^2
   + \ts 4 \ts \bigl( m^{}_{23} - m^{}_{13} \bigr)
   \bigl( m^{}_{21} - m^{}_{31} \bigr)  \, \geqslant \, 0 \ts ,
\]  
  where the inequality follows from the monotonicity of $M$ via 
  Example~\ref{ex:mon-con}.
  
  This implies $\sigma (M) \subset \RR$, and the formula for
  $\lambda^{}_{\pm}$ from Eq.~\eqref{eq:eig} shows that at most
  $\lambda^{}_{-}$ can be negative, because $\tr (M) \geqslant 1$ by
  Corollary~\ref{coro:trace}.  Also, when $\det (M) = 0$, the spectrum
  is $\sigma (M) = \{ 1, \tr (M) - 1, 0 \}$, which is non-negative.
  This establishes the claims on $M$.
  
  If $Q = (q^{}_{ij})^{}_{1\leqslant i,j \leqslant 3}$ is a Markov
  generator, its spectrum contains $0$, while the other two
  eigenvalues are given by
\begin{equation}\label{eq:eig-Q}
  \mu^{}_{\pm} \, = \, \myfrac{1}{2}
  \bigl( \tr (Q) \pm \sqrt{D}\, \bigr) ,
\end{equation}  
  where, in analogy to above, one finds
\[
  D \, = \, \bigl( q^{}_{11} - q^{}_{21} + q^{}_{23}
    - q^{}_{33} \bigr)^2
    + \ts 4 \ts \bigl( q^{}_{23} - q^{}_{13} \bigr)
    \bigl( q^{}_{21} - q^{}_{31} \bigr)
     \, \geqslant \, 0 \ts .
\]  
Here, the inequality follows via the monotonicity criterion from
Example~\ref{ex:gen-con}.
  
Consequently, all eigenvalues are real. They are then non-positive
because all eigenvalues of Markov generators have real part
$\leqslant 0$, compare \cite[Prop.~2.3{\ts}(1)]{BS}, as can also be
checked explicitly from \eqref{eq:eig-Q}, where
$\tr (Q) \leqslant 0 \leqslant \sqrt{D} \ts \leqslant \lvert \tr (Q)
\rvert$.
\end{proof}

Considering the convex combinations
\[
  M(\alpha) \, \defeq \, \alpha \ts E^{}_{1,1,2} +
  (1\nts \nts - \nts \alpha) E^{}_{2,3,3}
  \, = \, \begin{pmatrix} \alpha & 1\! - \nts\nts \alpha & 0 \\ 
    \alpha & 0 & 1\! - \nts \nts \alpha \\
        0 & \alpha & 1\! - \nts \nts \alpha 
  \end{pmatrix} \quad \text{with } \alpha \in (0,1) \ts , 
\]
in the notation of Table~\ref{tab:extremal}, one finds
$\tr ( M(\alpha)) = 1$ and
$\det (M(\alpha)) = - \alpha \ts (1 \nts -\alpha) < 0$. This shows
that cases with a simple negative eigenvalue exist. On the other hand,
Proposition~\ref{prop:mon-3} also means that the spectrum of a matrix
$M\in\cM^{}_{3,\preccurlyeq}$ is positive if and only if
$\det (M) > 0$.

\begin{coro}\label{coro:mon-3}
  Consider any\/ $M \in \cM^{}_{3,\preccurlyeq}$ with\/ $\det (M)<0$.
  Then, $M$ is neither em\-beddable nor can it have a monotone\/
  $n$-th Markov root for\/ $n$ even. In fact, $M$ has no Markov or,
  indeed, real square root at all, and\/ $M$ is not infinitely divisible.
\end{coro}

\begin{proof}
  Since $\det (M) < 0$, embeddability is ruled out immediately, and so
  is the existence of a real square root because $M$ has only a simple
  negative eigenvalue; compare \cite[Thm.~6.6]{Higham} or \cite{HL}.
  Consequently, $M$ cannot be infinitely divisible.

  Further, $B \in \cM^{}_{3,\preccurlyeq}$ with $B^{2m} = M$ for any
  $m\in\NN$ implies non-negative spectrum for $M$ due to
  $\sigma (B) \subset \RR$, which contradicts $\det (M) < 0$.
\end{proof}

Let us now analyse when a matrix $M\in\cM^{}_{3,\preccurlyeq}$
\emph{is} embeddable. For this, $M$ must be non-singular and thus, by
Corollary~\ref{coro:mon-3}, have positive spectrum. So, all
eigenvalues of $M$ must satisfy $0 < \lambda \leqslant 1$. Further,
$A = M \nts - \one$ is a generator that inherits monotonicity from $M$
as all off-diagonal elements of $T^{-1}\! A \ts\ts T$ are
non-negative.  Since the spectral radius of $A$ is $\varrho^{}_A < 1$,
\begin{equation}\label{eq:QausA}
    Q \, \defeq \, \log (\one + \nts A ) \, =
    \sum_{m=1}^{\infty} \myfrac{(-1)^{m-1}}{m} A^m
\end{equation}
converges and defines a real matrix with zero row sums and
$M \nts =\ee^Q$. As we are not interested in other types of solutions,
we now introduce the non-unital, real algebra
\begin{equation}\label{eq:algebra}
   \drei \, \defeq \, \{ B \in \Mat (3, \RR) : 
   \text{all row sums of $B$ are  $0$} \ts \} \ts ,
\end{equation}
which certainly contains the $Q$ from \eqref{eq:QausA}. Since positive
spectrum of $M$ means $\sigma (M) \subset ( 0, 1 ]$, all eigenvalues
of $Q$ are non-positive real numbers by the SMT.  It remains to
analyse the generator property and potential uniqueness of $Q$.

Let $q$ be the minimal polynomial of $A$. Its degree satisfies
$\deg (q) \in \{ 1,2,3 \}$ and equals the degree of the minimal
polynomial of $M$. Further, let
\[
  \alg (A) \, \defeq \, \langle A^m \nts :
     m \in \NN \ts \rangle^{}_{\RR}
\] 
be the real algebra spanned by the positive powers of $A$, which does
not contain $\one$, as one can check easily. In fact, one has
$\alg (A) = \langle A, A^2 \ts \rangle^{}_{\RR}$ by the
Cayley--Hamilton theorem, so $\alg (A)$ is a subalgebra of $\drei \!$
of dimension $\leqslant 2$. Indeed, we clearly get
 \begin{equation}\label{eq:dim-alg}
      \dim \bigl( \alg (A) \bigr) \, = \, \deg (q) - 1 \, \in \,
      \{ 0,1,2 \} 
 \end{equation}
 together with $ Q \in \alg (A)$ due to \eqref{eq:QausA}, which leads
 to different situations as follows.

\begin{prop}\label{prop:mon-embed}
  Let the matrix\/ $M \in \cM^{}_{3,\preccurlyeq}$ have a minimal
  polynomial of degree\/ $\leqslant 2$. Then, the following properties
  are equivalent.
\begin{enumerate}\itemsep=2pt
\item The spectrum of\/ $M$ is positive.
\item One has\/ $\det (M) > 0$.
\item $M$ is embeddable.
\item $M$ is embeddable with a monotone generator.
\end{enumerate}
If\/ $M$ is embeddable, there is precisely one monotone
generator\/ $Q$ with\/ $M = \ee^Q$, namely the one given
in Eq.~\eqref{eq:QausA}.
\end{prop}

\begin{proof}
  The implications $(4) \Rightarrow (3) \Rightarrow (2)$ are clear,
  while $(2) \Longleftrightarrow (1)$ follows from
  Proposition~\ref{prop:mon-3}.  It remains to show
  $(1) \Rightarrow (4)$, so assume
  $\sigma (M) \subset \RR_{+} \defeq \{x\in\RR: x>0\}$. As above, let
  $q$ be the minimal polynomial of $A=M\nts - \one$, which has the
  same degree as that of $M$.
  
  When $\deg (q) = 1$, since $0$ is always an eigenvalue of $A$, the
  only possibility is $q (x) = x$, hence $A=\nix$ and thus also
  $Q=\nix$ from \eqref{eq:QausA}, which gives the trivial case
  $\one = \exp(\nix)$, where $\one$ is the only matrix in
  $\cM^{}_{3,\preccurlyeq}$ with a minimal polynomial of degree
  $1$. By Remark~\ref{rem:one}, $Q=\nix$ is the only generator with
  $\ee^Q=\one$, which is trivially monotone.
  
  When $\deg (q) = 2$, we have $A\ne \nix$, hence $\tr (A) < 0$, so by
  \eqref{eq:dim-alg} we
  get $A^2 = - \alpha A$ for some $\alpha \in \RR$, where
  $\tr (A^2) > 0$ implies $\alpha > 0$. Here, $A=M \nts -\one$ is
  diagonalisable, with eigenvalues $0$ and $-\alpha > -1$, in line
  with $\sigma (M) \subset \RR_{+}$. Then, \eqref{eq:QausA} simplifies
  to $Q = - \frac{\log (1-\alpha)}{\alpha} A$, which is a positive
  multiple of $A$ and hence a monotone generator, so $M$ is embeddable
  as $M=\ee^Q$.

  To establish the uniqueness claim, consider any $Q{\ts}' \in \drei$
  such that $M = \ee^{Q{\ts}'}\!$, where $M$ is diagonalisable by
  assumption, hence also $Q{\ts}'$ by \cite[Fact~2.15]{BS}.  We then
  still have $[Q{\ts}', A]=\nix$, but not necessarily
  $Q{\ts}'\in\alg (A)$. Now, by \cite[Lemma~6.1]{BS}, there are two
  possibilities, namely
  $\dim \bigl(\alg (Q{\ts}' \ts ) \bigr) \in \{ 1,2 \}$.  Here, if the
  dimension is $2$, $Q{\ts}'$ must be simple, which is only possible
  if $Q{\ts}'$ has a complex-conjugate pair of (non-real)
  eigenvalues. But then, $Q{\ts}'$ cannot be monotone, by
  Proposition~\ref{prop:mon-3}. It remains to consider
  $ \dim \bigl(\alg (Q{\ts}' \ts ) \bigr) =1$, where we get
  $\alg (Q{\ts}' \ts) = \alg (A)$ from \cite[Lemma~6.1{\ts}(1)]{BS},
  hence $Q{\ts}' = a A$ for some $a>0$. By taking the determinant on
  both ends of $\ee^{Q{\ts}'} \! = M = \ee^Q$, which gives a positive
  number, one finds
 \[
   a \, \tr (A) \, = \, - \myfrac{\log (1-\alpha)}{\alpha}
   \, \tr (A) \ts .
 \]
 As $\tr (A) \ne 0$, this implies
 $a= - \frac{\log (1-\alpha)}{\alpha}$ and thus $Q{\ts}' \nts = Q$.
\end{proof}  
      
Let us pause to state the asymptotic behaviour of $M^n$ for the
embeddable matrices covered by Proposition~\ref{prop:mon-embed}.  In
the above notation, one trivially has $M^n=\one$ for all $n$ when
$\deg (q) = 1$, while a simple calculation gives
\[
   M_{\infty} \, \defeq \lim_{n\to\infty} M^n \, = \, \one \ts + \nts
   \lim_{n\to\infty} \! \frac{1 - (1-\alpha)^n}{\alpha} A 
   \, = \, \one + \frac{1}{\alpha} A
\]
for the more interesting case that $\deg (q) = 2$.   
\smallskip
      
When $\deg (q) = 3$, the situation becomes a little more complex.
Here, $A=M\nts - \one$ is always cyclic, with $0$ being a simple
eigenvalue. Then, we obtain $A^3 = r A + s A^2$ with
$r = \tr (M) - \det (M)-2$ and $s=\tr (A)$. Since
$\sigma (A) = \{ 0, \mu^{}_{+}, \mu^{}_{-} \}$, where $\mu^{}_{\pm}$
are negative numbers by Proposition~\ref{prop:mon-3}, we get
$r = - \mu^{}_{+} \mu^{}_{-} < 0$ and
$s = \mu^{}_{+} \nts + \mu^{}_{-} < 0$.  This remains correct when
$\mu^{}_{+} = \mu^{}_{-}$, where $A$ has a non-trivial Jordan normal
form (as it is cyclic). Note that, although $A$ is a \emph{real}
matrix, it is more convenient, and also completely consistent, to
always employ the \emph{complex} Jordan normal form of $A$ in our
arguments.
 
Let us first consider the case that $A$ is \emph{simple} (and hence
also diagonalisable). Here, we have
$-1 < \mu^{}_{-} \nts < \mu^{}_{+} \nts < 0$ together with
$\sigma (M) = \{ 1, 1+\mu^{}_{+}, 1+\mu^{}_{-} \}$.  As $A$ is cyclic,
any matrix $Q \in \drei$ with $M=\ee^Q$ must lie in
$\alg (A) = \RR [A] \cap \drei$, so $Q=\alpha A + \beta A^2$ for some
$\alpha,\beta \in \RR$, again by Frobenius' theorem. Then, the SMT
implies
\[
   \alpha \ts \mu^{}_{\pm} \nts + \beta \ts \mu^{2}_{\pm} \, = \,
   \log ( 1+ \mu^{}_{\pm}) \, \in \, \RR \ts ,
\]    
 which is an inhomogeneous system of linear equations for $\alpha$
 and $\beta$. As
\[
    \det \begin{pmatrix} \mu^{}_{+} & \mu^{2}_{+} \\
    \mu^{}_{-} & \mu^{2}_{-} \end{pmatrix} \, = \,
    \mu^{}_{+} \ts \mu^{}_{-}(\mu^{}_{-} \nts - \mu^{}_{+} ) 
    \, < \, 0 \ts ,
\] 
  we get a unique solution, which is given by
\begin{equation}\label{eq:albe}
    \alpha \,  = \, \frac{\mu^{2}_{-} \log (1+\mu^{}_{+}) -
      \mu^{2}_{+} \log (1 +\mu^{}_{-})}{ \mu^{}_{+} \mu^{}_{-}
      (\mu^{}_{-} \nts - \mu^{}_{+} ) } \, , \quad
    \beta \,  = \,  \frac{\mu^{}_{+} \log (1+\mu^{}_{-}) -
      \mu^{}_{-} \log (1 +\mu^{}_{+})}{ \mu^{}_{+} \mu^{}_{-}
      (\mu^{}_{-} \nts - \mu^{}_{+} ) } \ts .
\end{equation}
This shows that $M = \ee^Q$ has precisely one solution with
$Q\in \drei$, which must be the one from \eqref{eq:QausA}.  One can
check via the Taylor series that $\alpha > 0$ and $\beta < 0$, though
this is not sufficient to guarantee the generator or the monotonicity
property of $Q$.  So {far,} we have derived the following 
{constructive} result.
  
\begin{theorem}\label{thm:drei}      
  Let\/ $M\nts\in\cM_{{3,\preccurlyeq}}$ have simple spectrum
  with\/ $\det (M) > 0$, and set\/ $A = M \nts - \one$, with\/
  $\sigma (A) = \{ 0, \mu^{}_{+}, \mu^{}_{-} \}$ as above. Then, there
  is precisely one\/ $Q \in \drei$ such that\/ $M \nts = \ee^Q$,
  namely the matrix\/ $Q$ from~\eqref{eq:QausA}, which also
  satisfies\/ $\ts Q = \alpha A + \beta A^2$ with\/ $\alpha, \beta$
  from~\eqref{eq:albe}.
  
  Further, $Q$ is a generator, and\/ $M \nts =\ee^Q$, precisely when\/
  $(\alpha A + \beta A^2)^{}_{ij} \geqslant 0$ holds for all\/
  $i\ne j$, and\/ $Q$ is also monotone when the criterion from
  Example~\textnormal{\ref{ex:gen-con}} is satisfied by\/
  $\alpha A + \beta A^2$.  \qed
\end{theorem}
   
It remains to consider the case that $A$ is \emph{cyclic}, but not
simple.  Then, its eigenvalues are $0$ and $-1 < \mu < 0$, the latter
twice.  Also here, any solution of $M=\ee^Q$ with $Q\in\drei$ must be
of the form $Q=\alpha A + \beta A^2$, which implies the condition
$\alpha \mu + \beta \mu^2 = \log (1+\mu)$ by the SMT.  Using the
standard Jordan normal form of $A$, which must comprise the Jordan
block
$\bigl( \begin{smallmatrix} \mu & 1 \\ 0 &
    \mu \end{smallmatrix}\bigr)$ due to our assumption, one obtains
another condition from $\ee^Q = \one + A$, this time from the
superdiagonal element of the Jordan block, namely
$(1+\mu) (\alpha + 2 \beta \mu) = 1$. This results in the unique
solution
\begin{equation}\label{eq:albe-2}
  \alpha \, = \, 2 \, \myfrac{\log (1+\mu)}{\mu} - \myfrac{1}{1+\mu}
  \quad \text{and} \quad \beta \, = \, \myfrac{1}{\mu (1+\mu)} -
  \myfrac{\log (1+\mu)}{\mu^2} \ts .
\end{equation}
Note that \eqref{eq:albe-2} also follows from \eqref{eq:albe} by an
approximation argument of de L'Hospital type, via setting
$\mu^{}_{-} \nts = \mu = \mu^{}_{+} \nts + x$ and letting $x\to 0$.
So, the $Q$ from \eqref{eq:QausA} is once more the only solution for
$M=\ee^Q$ with $Q\in\drei$. Here, $\alpha > 0$ for $\mu$ sufficiently
large (approx. above $-0.7$) and $\beta < 0$, which is more
complicated than in the previous case. Nevertheless, we have the
following result.

\begin{coro}\label{coro:drei}
  Let\/ $M\nts\in\cM_{3,\preccurlyeq}$ be cyclic, but not simple.
  Then, $A = M \nts - \one$ has spectrum\/ $\sigma (A) = \{ 0, \mu \}$
  with\/ $-1 < \mu < 0$, where\/ $\mu$ has algebraic multiplicity\/
  $2$, but geometric multiplicity\/ $1$.  Further, all statements of
  Theorem~\textnormal{\ref{thm:drei}} remain true, this time with the
  coefficients\/ $\alpha, \beta$ from~\eqref{eq:albe-2}. \qed
\end{coro}  

Not all $M\in\cM_{3,\preccurlyeq}$ with positive determinant 
(and hence spectrum) can
be embeddable, as there are cases with a single $0$ in one position;
see \cite{Davies,Guerry,BS} and references therein for further
examples. For $d\geqslant 4$, the possibility of complex conjugate
pairs of eigenvalues increases the complexity of the treatment, which
is nevertheless possible with the recent results from \cite{CFR}.

\section{Uniqueness of embedding and further
    directions}\label{sec:unique}

  The explicit treatment of $\cM_{3,\preccurlyeq}$ in the previous
  section shows that some useful sufficient criteria for unique
  embeddability should be in store, such as the one stated in
  \cite[Sec.~2.3]{Higham} for Markov matrices with distinct positive
  eigenvalues. Let us first recall a classic result on the existence
  of a real logarithm {of a given matrix,} which can be found in 
  {several} places, for instance
  in \cite[Sec.~8.8.2]{Gant} or as \cite[Thm.~1]{Culver}.
   
\begin{fact}\label{fact:one-sol}
  For\/ $B\in \mathrm{GL} (d,\RR )$, the equation\/ $B = \ee^R$ has a
  solution\/ $R\in \Mat (d,\RR)$ if and only if every elementary
  Jordan block of\/ $B$ with an eigenvalue on the negative real axis
  occurs with even multiplicity. When\/ $B$ is diagonalisable, this
  simplifies to the condition that each eigenvalue of\/ $B$ on the
  negative real axis has even algebraic multiplicity. \qed
\end{fact}  

Any matrix $R\in \Mat (d,\RR)$ that solves $B=\ee^R$ is called a
\emph{real logarithm} of $B$.  When considering a non-singular Markov
matrix $M$, we are only interested in real logarithms of $M$ with zero
row sums, that is, in elements from the subalgebra
$\cA^{(d)}_{\, 0}\!\subset\Mat (d,\RR)$. This is justified by the
following observation. Suppose $\ee^R$ has unit row sums with a real
matrix $R$ that fails to have zero row sums.  Then, the set of
matrices $\ee^{tR}$ with unit row sums and $t\in\RR$ forms a discrete
subgroup of $\{ \ee^{tR}:t\in\RR\}\simeq \RR$. This is so because the
existence of an accumulation point with unit row sums, $t^{}_0$ say,
would result in $(1, \ldots, 1)^T$ being an eigenvector of $R$ with
eigenvalue $0$, which is a contradiction.

In analogy to the previous case with $d=3$, now for
$A\in\cA^{(d)}_{\, 0}\!$ with arbitrary $d\geqslant 2$, we define the
non-unital algebra
\[
    \alg (A) \, \defeq \, \langle A^m \nts : m \in \NN \rangle^{}_{\RR}
    \, = \, \langle A, A^2, \ldots , A^{d-1}\rangle^{}_{\RR}
    \, \subset \, \cA^{(d)}_{\, 0} ,
\]
where the second formulation follows from the Cayley--Hamilton
theorem in conjunction with the fact that $\cA^{(d)}_{\, 0}$ is
non-unital.
  
\begin{lemma}\label{lem:real}
  Let\/ $M \in \cM_d$ be cyclic and non-singular, and assume that\/
  $M$ possesses at least one real logarithm, according to
  Fact~\textnormal{\ref{fact:one-sol}}. Then, with\/
  $A = M\nts - \one$, any real logarithm\/ $R$ of\/ $M$ satisfies\/
  $R \in \alg (A)$.
\end{lemma}

\begin{proof}
  Clearly, we have $[R,M\ts ] = [R,A\ts ]=\nix$, so $M$ cyclic implies
  $R\in\RR [A\ts ]$ by Frobenius' theorem, and thus
  $R = \alpha^{}_{0}\ts \one + \sum_{n=1}^{d-1} \alpha^{}_{n} A^{n}$
  for some $\alpha^{}_{0}, \ldots , \alpha^{}_{d-1} \in \RR$. Hence,
  $R=\alpha^{}_{0}\ts \one + X$ with $X\in \cA^{(d)}_{\, 0}$, where
  $\ee^X$ then has unit row sums.  Consequently, all row sums of
  $\ee^R = \ee^{\alpha^{}_0} \ee^X$ equal $\ee^{\alpha^{}_0}$, which
  must be $1$. So, we get $\alpha^{}_{0} = 0$ and
  $R\in \alg (A) \subset \cA^{(d)}_{\, 0}$ as claimed.
\end{proof}

Now, we can extend the uniqueness result mentioned earlier to cyclic
matrices.  It is a variant of \cite[Thm.~2]{Culver}, but we give a
different and constructive proof that later leads to an effective (and
numerically stable) criterion for embeddability. It generalises what
we saw in Theorem~\ref{thm:drei} and Corollary~\ref{coro:drei}, and
also differs from the approach used in \cite{CFR}.

\begin{theorem}\label{thm:unique}
  Suppose\/ $M\in\cM_d$ is cyclic and has real spectrum. Then, $M$
  possesses a real logarithm\/ $R$, so\/ $M=\ee^R$, if and only if the
  spectrum of\/ $M$ is positive.  In this case, $R$ is unique, and is
  always an element of\/ $\alg (A) \subset \cA^{(d)}_{\, 0}$, where\/
  $A = M \nts - \one$.
\end{theorem}

\begin{proof}
  When $M$ is cyclic, no elementary Jordan block can occur twice, and
  the first implication follows from Fact~\ref{fact:one-sol}.  When
  $\sigma (M) \subset \RR_{+}$, due to Fact~\ref{fact:one-sol} and
  Lemma~\ref{lem:real}, all real logarithms of $M$ must lie in
  $\alg (A)$, and there is at least one $R\in\alg (A)$ with $\ee^R=M$,
  so we have $R = \sum_{i=1}^{d-1} \alpha^{}_{i} A^i$ for some
  $\alpha^{}_{1}, \ldots, \alpha^{}_{d-1} \in \RR$. It remains to
  establish uniqueness.

  First, assume that $A$ is \emph{simple}. As $A$ is a generator, this
  means $\sigma (A) = \{ 0, \mu^{}_{1}, \ldots , \mu^{}_{d-1}\}$, with
  distinct $\mu^{}_{i} \in (-1,0)$ due to our assumptions. Then, since
  all $\alpha^{}_{i}$ and $\mu^{}_{j}$ are real, the SMT implies the
  $d- \nts 1$ identities
\begin{equation}\label{eq:alphas}
    \sum_{\ell=1}^{d-1} \alpha^{}_{\ell} \, \mu^{\ell}_{i} \, = \,
    \log (1 + \mu^{}_{i}) \ts , \qquad 
    1  \leqslant i \leqslant d- \nts 1 \ts .
\end{equation}
They constitute an inhomogeneous system of linear equations for the
$\alpha^{}_{i}$ with the matrix
\begin{equation}\label{eq:B-mat}
   B \, = \,  \begin{pmatrix} \mu^{}_{1} & \mu^{2}_{1} & \cdots & 
     \mu^{d^{\vphantom{I}}-1}_{1} \\
     \mu^{}_{2} & \mu^{2}_{2} & \cdots & 
     \mu^{d^{\vphantom{I}}-1}_{2} \\
     \vdots & \vdots & & \vdots \\ \mu^{}_{d-1} & \mu^{2}_{d-1} &
     \cdots & \mu^{d^{\vphantom{I}}-1}_{d-1} \end{pmatrix} .
\end{equation}
Since
$\det (B) = \left( \prod_i \mu^{}_{i} \right) \prod_{k>\ell}
(\mu^{}_{k} - \mu^{}_{\ell})$ by an obvious variant of the standard
Vandermonde determinant formula, $B$ is invertible when $A$ is simple,
and \eqref{eq:alphas} has a unique real solution.
    
When $A$ is \emph{cyclic}, but not simple, the appearance of
non-trivial Jordan blocks necessitates a more refined argument.
Clearly, as $A$ is a generator and also cyclic, $0$ is a simple
eigenvalue of $A$ by \cite[Prop.~2.3{\ts}(2)]{BS}.  Let
$\mu \in (-1,0)$ be any of the other eigenvalues, say with algebraic
multiplicity $m$. When $m=1$, we get one condition from the SMT, and
nothing else is needed. So, assume $m\geqslant 2$. As $A$ is cyclic,
the geometric multiplicity of $\mu$ is $1$, and the corresponding
Jordan block in standard form is
$\JJ^{}_{\mu} = \mu \ts \one_m + N_m$, where $N_m$ is the nilpotent
matrix with entries $1$ on the first superdiagonal and $0$ everywhere
else. It satisfies $N^{m}_{m} = \nix$, while $N^{k}_{m}$, for
$1\leqslant k < m$, has entries $1$ on the $k$-th superdiagonal and
$0$ elsewhere.  In this case, we get only \emph{one} condition from
the SMT, namely
\begin{equation}\label{eq:lam-sum}
     \sum_{\ell =1}^{d-1} \alpha^{}_{\ell} \, \mu^{\ell} 
        \, = \, \log (1 + \mu) \ts ,
\end{equation}
as in \eqref{eq:alphas}, while we need $m-1$ independent further
ones. They will come from derivatives of \eqref{eq:lam-sum}, which
needs a justification as follows.
  
First, from $\ee^R = \one + A$, one concludes that we must have
\begin{equation}\label{eq:comp}
     \exp \biggl( \,\sum_{\ell=1}^{d-1} \alpha^{}_{\ell} \,
     \JJ^{\ts \ell}_{\mu}\biggr) 
     \, = \: \one^{}_{m} + \JJ^{}_{\mu}
     \: = \: (1+\mu)\ts\ts \one^{}_{m} + N^{}_{m} \ts .
\end{equation}
Now, defining the polynomial
$\phi (x) = \sum_{\ell=1}^{d-1} \alpha^{}_{\ell} \, x^{\ell}$ and the
function $\psi (x) = \ee^{\ts\phi (x)}\!$, one can employ the standard
method from \cite[Sec.~1.2.1]{Higham} to calculate the exponential in
\eqref{eq:comp} as
\[
  \psi \bigl( \JJ^{}_{\mu}\bigr) \, = \: \psi (\mu) \, \one^{}_{m} \, +
  \sum_{k=1}^{m-1} \myfrac{\psi^{(k)} (\mu)}{k\ts !} \ts N^{k}_{m}
  \ts ,
\]
where $\psi^{(k)}$ denotes the $k$-th derivative of $\psi$. As
$\psi (\mu) = 1+\mu $ by \eqref{eq:lam-sum}, a comparison with
\eqref{eq:comp} leads to the conditions
$\psi^{(k)} (\mu) = \delta^{}_{k,1} $ for
$1 \leqslant k \leqslant m-1$, noting that $1+\mu\ne 0$.  Iterating
the product rule on $\psi = \ee^{\ts \phi}$ and inserting
\eqref{eq:lam-sum} then results in
\begin{equation}\label{eq:derive}
   \myfrac{\dd^k}{\dd \mu^k}
   \, \sum_{\ell=k}^{d-1} \alpha^{}_{\ell} \, \mu^{\ell}
   \, = \, \phi^{(k)} (\mu) \, = \,
   \myfrac{(-1)^{k-1} (k-1) !}{(1+\mu)^k} \, = \,
   \myfrac{\dd^k}{\dd \mu^k} \ts \log (1+\mu) \ts ,
\end{equation}
which shows that the additional conditions on the $\alpha^{}_{\ell}$
emerge via derivatives of the fundamental relation \eqref{eq:lam-sum}.
It remains to check when we obtain a \emph{unique} solution
$\alpha^{}_{1}, \ldots , \alpha^{}_{d-1}$ this way.
 
Here, the matrix $B$ that generalises \eqref{eq:B-mat} is specified by
an $n$-tuple
$\bigl( (\mu^{}_{1}, m^{}_{1}), \ldots , (\mu^{}_{n}, m^{}_{n})
\bigr)$ of distinct non-zero eigenvalues $\mu^{}_i$ of $A$ with their
algebraic multiplicities $m^{}_i$, where $n\leqslant d- \nts 1$ and
$m^{}_{1}\nts + \ldots + m^{}_{n} = d-\nts 1$. A pair
$(\mu^{}_{i}, m^{}_{i})$ is responsible for $m^{}_{i}$ rows of $B$ of
length $d-\nts 1$, where the first derives from \eqref{eq:lam-sum},
followed by $m^{}_{i} - \nts 1$ rows induced by
\eqref{eq:derive}. Here, each new row emerges from the previous one by
differentiation with respect to $\mu^{}_{i}$. The resulting matrix is
a variant of the \emph{confluent Vandermonde matrix}, see
\cite[Sec.~22.2]{Higham-2} or \cite{Luther}, which is known from
Hermite interpolation. It is invertible if and only if the
$\mu^{}_{i}$ are distinct.

To substantiate the latter claim, define the sequence
$(\gamma^{}_{n})^{}_{n\in\NN}$ by
\[
    \gamma^{}_{n} \, = \, \det \begin{pmatrix}
     1 & 1 & 1 & \cdots & 1 \\
     1 & 2 & 4 & \cdots & 2^{n-1} \\
     \vdots & \vdots & \vdots & & \vdots \\
     1 & n & n^2 & \ldots & n^{n-1} \end{pmatrix} ,
\]
which starts as $1, 1, 2, 12, 288, 34{\ts\ts}560$; compare
\cite[A{\ts}000{\ts\ts}178]{OEIS}. Then, one finds
\[
    \det (B) \, = \, \biggl( \,\prod_{i=1}^{n} \mu^{m_i}_{i} 
    \ts \gamma^{}_{m_i} \biggr) \prod_{k>\ell}
    \bigl( \mu^{}_{k} - \mu^{}_{\ell} \bigr)^{m^{}_k m^{}_{\ell}},
\]
compare \cite[Exc.~22.6]{Higham-2}, which reduces to the determinant
formula stated previously for the special case
$m^{}_{1} = \ldots = m^{}_{n} = 1$. Since the $\mu^{}_{i}$ are
distinct, $\det (B)\ne 0$ is clear, and the claimed uniqueness
follows.
\end{proof}

The benefit of this approach is that one can calculate $B^{-1}$ and
thus determine the coefficients purely from the eigenvalues of $A$. In
particular, the unique $R$ is a generator if and only if
$\sum_{i=1}^{d-1}\alpha^{}_{i} A^{i}$ satisfies the corresponding
conditions.

\begin{remark}
  The Vandermonde matrix and its inverse is well known from Lagrange
  interpolation theory. Therefore, in minor modification of the
  results from \cite[Sec.~0.9.11]{HJ}, one finds the matrix elements
  of the inverse of $B$ from \eqref{eq:B-mat} as
\[
  \bigl( B^{-1}\bigr)_{ij} \, = \, \frac{(-1)^{i-1} S^{(d-2)}_{d-1-i}
    \bigl( \{ \mu^{}_{1} , \ldots , \mu^{}_{d-1} \} \setminus \{
    \mu^{}_{j} \}\bigr)}{\mu^{}_{j} \prod_{k\ne j} (\mu^{}_{k} -
    \mu^{}_{j}) }
\]
where $S^{(n)}_{m} \bigl( \{ a^{}_{1}, \ldots , a^{}_{n} \} \bigr)$ is
the elementary symmetric polynomial as defined by
$S^{(n)}_{0} \equiv 1$,
$S^{(n)}_{1} \bigl( \{ a^{}_{1}, \ldots , a^{}_{n} \} \bigr) =
a^{}_{1} \nts + \ldots + a^{}_{n}$,
$S^{(n)}_{2} \bigl( \{ a^{}_{1}, \ldots , a^{}_{n} \} \bigr) =
\sum_{i<j} a^{}_{i} a^{}_{j}$ and so on, up to the final one, which is
$S^{(n)}_{n} \bigl( \{ a^{}_{1}, \ldots , a^{}_{n} \} \bigr) =
\prod_{i} a^{}_{i}$. 

With a little more effort, this formula can be extended to the cyclic 
situation as well; see \cite{Luther} for a constructive approach to 
$B^{-1}$ in this case.  \exend
\end{remark}

The above result leads to the following sufficient criterion for
unique embeddability.

\begin{coro}\label{coro:unique}
  Let\/ $M\in\cM_d$ be cyclic and have real spectrum,
  $\sigma (M) \subset \RR$.  Then, $M$ has a real logarithm if and
  only if\/ $\sigma (M) \subset \RR_{+}$.

  In this case, the spectral radius of\/ $A=M\nts - \one$ is\/
  $\varrho^{}_{A} < 1$, and there is at most one Markov generator\/
  $Q$ such that\/ $M = \ee^Q$.  The only choice is\/
  $ Q = \log (\one + A) \in \alg (A) \subset \cA^{(d)}_{\, 0} \!$,
  calculated with the standard branch of the matrix logarithm as a
  convergent series. In particular, $M$ is embeddable if and only if
  the matrix\/ $\log (\one + A)$ is a generator.
\end{coro}

\begin{proof}
  The first claim follows from Theorem~\ref{thm:unique}. If
  $\sigma (M) \subset \RR_{+}$, all eigenvalues of $A$ lie in the
  half-open interval $(-1, 0\ts ]$, so $\varrho^{}_{A}<1$ is clear. By
  Theorem~\ref{thm:unique}, as $M$ is cyclic with
  $\sigma (M) \subset (0,1]$, there is precisely one real matrix $R$
  with $M=\ee^R$. Due to $\varrho^{}_{A} < 1$, the series
\[
   \log (\one + A) \, =  \sum_{m=1}^{\infty} \myfrac{(-1)^{m-1}}{m} A^m
\]  
converges. The limit is then an element of $\alg (A)$, because this
algebra is a closed subset of $\Mat (d,\RR)$. So, we get
$R = \log (\one + A) \in \alg (A)$ in this case, which has zero row 
sums, but need not be a generator.
\end{proof}

One can go beyond this result, but care is required with the branches 
of the complex logarithm; see \cite[Ch.~11]{Higham} for background
and \cite{CFR} for recent progress in this direction. 

\begin{remark}
  The results from Theorem~\ref{thm:drei} and
  Corollary~\ref{coro:drei} apply to all cyclic matrices $M\in\cM_3$
  with positive spectrum in that the embeddability of $M$ can most
  easily be verified via testing whether $\alpha A + \beta A^2$ is a
  generator, where $A=M\nts - \one$ and $\alpha,\beta\in\RR$ are the
  numbers from Eq.~\eqref{eq:albe}, if $M$ is simple, or from
  Eq.~\eqref{eq:albe-2} otherwise.  \exend
\end{remark}

Clearly, uniqueness results have interesting consequences on the
structure of Markov roots, as can be seen in the following refinement
of \cite[Ex.~3.9]{BS}.

\begin{example}
  The two-dimensional Markov matrix 
\[
  M \, = \, \begin{pmatrix} \tfrac{3}{4} & \tfrac{1}{4} \\
  \tfrac{1}{2} & \tfrac{1}{2} \end{pmatrix} 
\]
is uniquely embeddable by Lemma~\ref{lem:Kendall}, so $M\nts = \ee^Q$
with a unique generator $Q$. Nevertheless, as follows from a simple
calculation, it has precisely two Markov square roots, namely
\[
   M^{}_{1} \, = \, \begin{pmatrix} \tfrac{5}{6} & \tfrac{1}{6} \\
   \tfrac{1}{3} & \tfrac{2}{3} \end{pmatrix} \quad \text{and} \quad 
   M^{}_{2} \, = \, \begin{pmatrix} \tfrac{1}{2} & \tfrac{1}{2} \\
     1 & 0 \end{pmatrix} .
\]
Of these, $M^{}_{1} = \exp \bigl( \frac{1}{2} \ts Q \bigr)$ is
embeddable, while $M^{}_{2}$ is not. So, in the embeddable case, there
is always at least one Markov $n$-th root for every $n\in\NN$ of the
form $\exp \bigl( \frac{1}{n} \ts Q \bigr)$, but there can still be
others. A uniqueness result for the embedding of $M$ then means that,
among all Markov roots, there is precisely one sequence of embeddable
Markov $n$-th roots of $M$.  \exend
\end{example}

The set $\emb$ of embeddable matrices is a relatively closed subset of
$\{ M\in\cM_d : \det (M) >0\}$ by \cite[Prop.~3]{King}, but (for
$d\geqslant 2$) it is \emph{not} a closed subset of $\cM_d$. The
closure of $\emb$ is still a subset of the infinitely divisible
elements of $\cM_d$, and it is a natural question which matrices lie
on the boundary, which we denote by $\partial\emb$. Clearly, there can
be embeddable cases, such as
$\left(\begin{smallmatrix} 1 & 0 \\ 1-\alpha & \alpha
  \end{smallmatrix}\right)$ or $\left(\begin{smallmatrix} \alpha &
    1-\alpha \\ 0 & 1 \end{smallmatrix}\right)$ for $d=2$ and
$0 < \alpha \leqslant 1$, as well as (singular) idempotent ones, such
as $M_{\infty} = \lim_{t\to\infty}\ee^{t\ts Q}$ for any generator
$Q\ne \nix$. For $d>2$, there are further possibilities. Any
$M\in \partial\emb$ satisfies $M=\lim_{n\to\infty}\ee^{Q_n}$ for some
sequence $(Q_n)^{}_{n\in\NN}$ of generators, which implies
$\lim_{n\to\infty} \ee^{\tr (Q_n)} = \det (M) \geqslant 0$.

When $M\in \partial\emb$ has $\det(M) > 0$, it is embeddable by 
Kingman's infinite
divisibility criterion. Alternatively, a positive determinant implies
that the sequence $\bigl(\tr(Q_n)\bigr)_{n\in\NN}$ converges and is
thus bounded. Since all diagonal elements of a generator are
non-positive, they are bounded as well, hence also all elements of the
$Q_n$ due to the vanishing row sums.  By a standard compactness
argument, there is thus a subsequence $(Q_{n_i})^{}_{i\in\NN}$ such
that $\lim_{i\to\infty}Q_{n_i} = Q$ is a generator with $M=\ee^Q$ as
expected, and $M$ lies in the Markov semigroup $\{ \ee^{t\ts Q}:
t\geqslant 0\}$.

When $\det (M)=0$, the sequence $\bigl(\tr(Q_n)\bigr)_{n\in\NN}$ must
be (negatively) unbounded, and we may assume that, at least at one
off-diagonal position, the entries of the $Q_n$ are (positively)
unbounded. When $d=2$, this suffices to show that $M$ is a singular
idempotent. Already for $d=3$, the situation becomes more complex,
since one can have a limiting $M$ with $\det (M) = 0$ that is not
an idempotent, by considering
\[
  \exp \begin{pmatrix} -a\! -\! b & a & b \\ c & -c\! -\! d & d \\
    e & f & -e\! - \!\nts f \end{pmatrix}
\]
for $a,b,\ldots,f \geqslant 0$. Then,
fixing $b, \ldots, f$ at generic values and letting $a\to\infty$
produces such examples, and similarly for various other choices.  When
$d=4$, one can have mixtures in block matrix form, such as
\[
  M \, = \, \begin{pmatrix} 1 & 0 \\ 1\! - \nts\nts \alpha & \alpha
    \end{pmatrix} \oplus \begin{pmatrix} \beta & 1\! - \nts\nts\beta \\
    \beta & 1\! -\nts\nts\beta \end{pmatrix}
\]
for $\alpha\in (0,1)$ and $\beta \in [\ts 0,1]$, which is singular but
not idempotent.

It seems worthwhile to characterise the boundary more completely, for
instance by relating semigroups with reducible generators to
properties of the boundary, which we leave as an open problem at this
point. It is clear though that Markov idempotents and their connection
with blockwise equal-input matrices will be important here.  More
generally, a simplified systematic treatment of the embeddability
problem for $d\leqslant 4$ would be helpful in view of the
applications in phylogeny; see \cite{CFR} for recent progress in this
direction. Finally, even some of the elementary questions become much
harder in the situation of countable states, where many new phenomena
occur; see \cite{ER} and references therein for some recent results.
\smallskip

\section*{Acknowledgements}

It is our pleasure to thank Frederic Alberti and Martin M\"{o}hle for
valuable discussions and suggestions, as well as Tanja Eisner, Agnes
Radl and Peter Taylor for helpful comments on the manuscript. 
{Further,
we thank two anonymous referees for their thoughtful comments, which
helped us to improve the presentation.} This
work was supported by the German Research Foundation (DFG), within the
CRC 1283 at Bielefeld University, and by the Australian Research
Council (ARC), via Discovery Project DP 180{\ts}102{\ts}215.


\smallskip

\begin{thebibliography}{99}
\small

\bibitem{BS}
M.~Baake and J.~Sumner,
Notes on Markov embedding,
\textit{Lin.\ Alg.\ Appl.} \textbf{594} (2020) 262--299;
\texttt{arXiv:1903.08736}.

\bibitem{CFR}
M.~Casanellas, J.~Fern\'{a}ndez-S\'{a}nchez and J.~Roca-Lacostena,
The embedding problem for Markov matrices,
\textit{preprint} (2020);
\texttt{arXiv:2005.00818}.
  
\bibitem{Culver}
W.J.~Culver,
On the existence and uniqueness of the real logarithm
of a matrix, \textit{Proc.\ Amer.\ Math.\ Soc.} 
\textbf{17} (1966) 1146--1151.

\bibitem{Daley}
D.J.~Daley,
Stochastically monotone Markov chains,
\textit{Z.\ Wahrscheinlichkeitsth.\ Verw.\ Geb.}
\textbf{10} (1968) 305--317.

\bibitem{Davies}
E.B.~Davies,
Embeddable Markov matrices,
\textit{Electronic J.\ Probab.} \textbf{15} (2010)
paper 47, 1474--1486;
\texttt{arXiv:1001.1693}.

\bibitem{ER}
T.~Eisner and A.~Radl,
Embeddability of real and positive operators,
\textit{Lin.\ Multilin.\ Alg.}, in press; \newline
\texttt{arXiv:2003:08186}.

\bibitem{Elfving}
G.~Elfving,
Zur Theorie der Markoffschen Ketten,
\textit{Acta Soc.\ Sci.\ Fennicae}
\textbf{A2} (1937) 1--17.

\bibitem{Feller}
W.~Feller,
\textit{An Introduction to Probability Theory and Its Applications.
  Vol.~II}, 2nd ed., Wiley, New York (1971).

\bibitem{Jeremy}
J.~Fern\'{a}ndez-S\'{a}nchez, J.G.~Sumner, P.D.~Jarvis and
M.D.~Woodhams, Lie Markov models with purine/pyrimidine symmetry,
\textit{J.\ Math.\ Biol.} \textbf{70} (2015) 855--891;
\texttt{arXiv:1206.1401}.

\bibitem{Gant}
F.R.~Gantmacher,
\textit{Matrizentheorie}, Springer, Berlin (1986).

\bibitem{GMMS}
F.~G\"{o}nd\H{o}cs, G.~Michaletzky, T.F.~M\'{o}ri and G.J.~Sz\'{e}kely, 
A characterisation of infinitely divisible Markov chains
  with finite state space,
\textit{Ann.\ Univ.\ Sci.\ Budap.\ R.\ E\"{o}tv\"{o}s Nom.}
\textbf{27} (1985) 137--141.

\bibitem{Guerry}
M.-A.~Guerry,
Sufficient embedding conditions for three-state discrete-time 
Markov chains with real eigenvalues,
\textit{Lin.\ Multilin.\ Alg.} \textbf{67} (2019) 106--120.

\bibitem{Higham-2}
N.J.~Higham,
\textit{Accuracy and Stability of Numerical Algorithms},
2nd ed., SIAM, Philadelphia, PA (2002).

\bibitem{Higham}
N.J.~Higham, 
\textit{Functions of Matrices: Theory and Computation},
SIAM, Philadelphia, PA (2008).

\bibitem{HL}
N.J.~Higham and L.~Lin,
On $p${\ts\ts}th roots of stochastic matrices,
\textit{Lin.\ Alg.\ Appl.} \textbf{435} (2011) 448--463.

\bibitem{HM}
G.~H\"{o}gn\"{a}s and A.~Mukherjea,
\textit{Probability Measures on Semigroups:\ Convolution
Products, Random Walks and Random Matrices},
2nd ed., Springer, New York (2011).

\bibitem{HJ}
R.A.~Horn and C.R.~Johnson,
\textit{Matrix Analysis}, 2nd ed.,
Cambridge University Press, Cambridge (2013).

\bibitem{KK}
J.~Keilson and A.~Kester,
Monotone Markov matrices and monotone Markov processes,
\textit{Stoch.\ Proc.\ Appl.} \textbf{5} (1977) 231--241.

\bibitem{Ki}
M.~Kijima,
\textit{Markov Processes for Stochastic Modeling},
Springer, Boston, MA (1997).

\bibitem{King}
J.F.C.~Kingman,
The imbedding problem for finite Markov chains,
\textit{Z.\ Wahrscheinlichkeitsth.\ Verw.\ Geb.}
\textbf{1} (1962) 14--24.

\bibitem{Lang}
S.~Lang,
\textit{Algebra},
3rd ed., Addison-Wesley, Reading, MA (1993).

\bibitem{Lind}
B.H.~Lindqvist,
Monotone and associated Markov chains, with applications
to reliability theory,
\textit{J.\ Appl.\ Prob.} \textbf{24} (1987) 679--695.

\bibitem{Luther}
U.~Luther and K.~Rost,
Matrix exponentials and inversion of confluent
Vandermonde matrices,
\textit{Electr.\ Trans.\ Num.\ Anal.\ (ETNA)}
\textbf{18} (2004) 91--100.

\bibitem{MM}
M.~Marcus and H.~Minc,
\textit{A Survey of Matrix Theory and Matrix Inequalities},
reprint, Dover, New York (1992).

\bibitem{Norris}
J.R.~Norris,
\textit{Markov Chains}, reprint,
Cambridge University Press, Cambridge (2005).

\bibitem{OEIS}
N.J.A.~Sloane,
\textit{The On-Line Encyclopedia of Integer Sequences},
\texttt{https://oeis.org}.

\bibitem{Steel}
M.~Steel,
\textit{Phylogeny{\ts\ts}---{\ts}Discrete and Random Processes
in Evolution}, SIAM, Philadelphia, PA (2016).

\bibitem{J2}
J.~Sumner,
Multiplicatively closed Markov models must form Lie algebras,
\textit{ANZIAM J.} \textbf{59} (2017) 240--246;
\texttt{arXiv:1704.01418}.

\bibitem{J1}
J.G.~Sumner, J.~Fern\'{a}ndez-S\'{a}nchez and P.D.~Jarvis,
Lie Markov models,
\textit{J.\ Theor.\ Biol.} \textbf{298} (2012) 16--31;
\texttt{arXiv:1105.4680}.

\bibitem{Web}
R.~Webster,
\textit{Convexity},
Oxford University Press, New York (1994).

\bibitem{YR}
Z.~Yang and B.~Ranala,
Molecular phylogenetics: principles and practice,
\textit{Nature Rev.\ Genetics} \textbf{13} (2012) 303--314.

\end{thebibliography}
\end{document}